\documentclass[11pt,a4paper]{article}
\def\TeXHOME{.}

\usepackage{graphicx}
\usepackage{amsmath,amsfonts,amsthm}
\usepackage{amssymb,latexsym}
\usepackage{fixmath}
\usepackage{mathrsfs,amsbsy}
\usepackage{dsfont}
\usepackage{enumerate}

\def\wtd{\widetilde}
\def\what{\widehat}

\DeclareMathOperator{\diag}{diag}

\DeclareMathOperator{\subspan}{span}

\DeclareMathOperator{\UI}{ui}

\DeclareMathOperator{\F}{F}

\newtheorem{theorem}{Theorem}[section]
\newtheorem{lemma}{Lemma}[section]

\theoremstyle{definition}

\newtheorem{remark}{Remark}[section]
\newtheorem{example}{Example}[section]

\numberwithin{equation}{section}
\numberwithin{figure}{section}
\numberwithin{table}{section}

\allowdisplaybreaks

\usepackage{hyperref}
\usepackage{graphicx}
\usepackage{pstricks}
\usepackage{epstopdf}

\hypersetup{colorlinks}
\usepackage{subfig}
\usepackage{enumerate}

\usepackage{comment}
\usepackage{xcolor}
\usepackage[normalem]{ulem}

\usepackage{mathdots}

\def\ol{\overline}
\def\scrL{\mathscr{L}}

\newcommand\typ[2][R]{\textbf{#1#2}}
\def\ii{\mathrm{i}}

\DeclareMathOperator{\pinv}{+}
\DeclareMathOperator{\nullspace}{\mathcal{N}}
\DeclareMathOperator{\vectorize}{vec}

\title{Relative Perturbation Theory for Quadratic Hermitian Eigenvalue Problem}

\author{
Peter Benner\thanks{
Max Planck Institute for Dynamics of Complex Technical Systems,
Sandtorstra{\ss}e 1, 39106 Magdeburg, Germany.
E-mail: {\tt benner@mpi-magdeburg.mpg.de}.
}
\and
Xin Liang\thanks{
Max Planck Institute for Dynamics of Complex Technical Systems,
Sandtorstra{\ss}e 1, 39106 Magdeburg, Germany.
E-mail: {\tt liangxinslm@tsinghua.edu.cn}.
}
\and
Suzana Miodragovi{\'c}\thanks{
Department of Mathematics, University of Osijek, Trg Ljudevita Gala 6,
31000 Osijek, Croatia.
E-mail: {\tt ssusic@mathos.hr}.
}
\and
Ninoslav Truhar\thanks{
Department of Mathematics, University of Osijek, Trg Ljudevita Gala 6,
31000 Osijek, Croatia.
E-mail: {\tt ntruhar@mathos.hr}.
}
}
\date{\today}

\begin{document}

\maketitle

\begin{abstract}

In this paper, we derive new relative perturbation bounds for eigenvectors and eigenvalues for regular quadratic eigenvalue problems of the form
$(\lambda^2 M + \lambda C  + K)x = 0$,
where $M$ and $K$ are nonsingular Hermitian matrices and $C$ is a general Hermitian matrix.
These results are based on new relative perturbation bounds for an equivalent regular Hermitian matrix pair $A-\lambda B$.
The new bounds can be applied to quadratic eigenvalue problems appearing in many relevant applications,
such as mechanical models with indefinite damping.
The quality of our bounds is demonstrated by several numerical experiments.
\end{abstract}



\section{Introduction}
The quadratic eigenvalue problem (QEP) is to find scalars $\lambda$ and nonzero vectors $x$ satisfying
\begin{equation}\label{intro:QEP}
	(\lambda^2 M + \lambda C  + K) x=0,
\end{equation}
where $M$, $C$ and $K$ are $n\times n$ complex matrices. 
The solution to the QEP is required in many
applications arising in the dynamic analysis of structural mechanical
and acoustic systems, in electronic circuit simulation, in fluid
mechanics, in modeling micro-electronic-mechanical systems (MEMS), and so
on. The number of applications of the QEP is constantly growing.
In \cite{glr:82,lanc:66}, the theoretical
background of the QEP and other polynomial eigenvalue problems is extensively
studied; in \cite{TissMeer2001} a survey of
applications, mathematical properties, and a variety of numerical
solution techniques for the QEP is given.

This paper aims to build relative perturbation bounds for
eigenvalues and eigenspaces of the QEP \eqref{intro:QEP}, where
$M$, $C$, and $K$ are allowed to be Hermitian indefinite matrices, and
$C$ can be even singular.

In general, the perturbation theory of matrix or operator
eigenvalue problems can be divided into two major branches. One
belongs to the so-called standard or absolute perturbation theory
which can be found in many well-known textbooks, e.g., \cite{Kato1966,
  Godunov, Baumgartel, Bhatia, Parlett, Stewart1990}.
The other, the so-called
relative perturbation theory, has became a very active research
area since the late 1980s. Some important results of the relative perturbation theory can be
found in \cite{DemKah, BarDem, VesDem}. The development of such a kind of
theory goes back to as early as Kahan's technical report \cite{Kahan}
in 1966.

Regarding the perturbation theory for the QEP in a general setting,
standard or absolute perturbation bounds are given, for example, in
\cite{lili:15:hyperbolic, Veselic2011, TissMeer2001}, but to the
authors' knowledge there are no relative perturbation bounds for
eigenvalues and especially for eigenspaces of QEPs where all
three coefficient matrices are Hermitian without any further assumptions.

Some results on relative perturbations of the QEP can be found in
\cite{TruMio}. There one can find bounds for eigenvectors and
eigenvalues of the QEP \eqref{intro:QEP}, where the matrices $M$, $C$ and
$K$ are positive definite Hermitian matrices and the
condition
\[
	(x^HCx)^2 - 4(x^HMx)( x^H K x) > 0,\quad \forall x \in \mathbb{C}^n, x \neq 0
\]
is satisfied,  which means that the
corresponding QEP is hyperbolic.
The hyperbolic QEP is investigated, e.g., in \cite{TissMeer2001, lili:15:hyperbolic}.

In this paper we will derive similar bounds for the more general case,
that is, the QEP will be regular and $C$ can be any Hermitian matrix.  


Such QEPs are interesting because many of the instabilities in mechanical systems are caused by a negative definite or indefinite damping matrix $C$, see for example \cite{petit:2012, model, mahdi:2009, Kao2000,kleim:2009, Damm:2011}. Therefore, we will illustrate our
results using numerical experiments motivated by these problems.


Perturbation bounds for the regular QEP will be derived by using linearization and perturbation bounds for any regular matrix pair
 $A-\lambda B$, where $A$ and $B$ are Hermitian matrices, given in Section~\ref{Sec3}.



There is a vast amount of material in perturbation theory
(covering absolute and relative perturbation results) for definite
matrix pairs. Here, we will list some of those results related to
eigenspace perturbations, where the distance between two
eigensubspaces is measured by some trigonometric function of
the angle operator $\Theta$ associated with the eigensubspaces
$\mathcal X=\subspan(X)$ and $\mathcal{Y}=\subspan(Y)$ of the unperturbed and
perturbed matrix pairs. The angle operator $\Theta$ is defined
by
\[
	\Theta(\mathcal{X}, \mathcal{Y}) = \arcsin(P_{\mathcal{X}} -
P_{\mathcal{Y}}),
\]
where $P_{\mathcal{X}}$ and $P_{\mathcal{Y}}$
are orthogonal projections onto the subspaces $\mathcal X$ and
$\mathcal{Y}$, of the same dimension (see, e.g., \cite{Stewart1990}).
The eigenvalues of the matrix $\Theta$ represent canonical angles.
One can find results for standard perturbation theory and
relative perturbation bounds for eigenspaces in
\cite{Stewart1990, Davis1970} and in \cite{Grubivsi'c2010, Grub2012,
  TruMio}, respectively.
For the case when $B$ is Hermitian positive definite, in
\cite{Grubivsi'c2010, Grub2012}, the authors define the angle operator $\Theta$
in the matrix-dependent scalar product $\langle x,y \rangle_B=y^T B
x$, for $x,y \in \mathbb{R}^n$. 
A similar approach is not possible for the case when $B$ is Hermitian indefinite nonsingular.
For that case, as it has been shown in \cite{TruMio},  it is
possible to measure the distance between two subspaces, using the fact
that $\|\sin \Theta(\mathcal{X}, \mathcal{Y})\|\rightarrow 0$ if and
only if $\| Y^H B X\| \rightarrow 0$. The same idea will be used in this paper.

This paper is organized as follows.
In Section~\ref{Sec3} we derive our relative perturbation $\sin
\Theta$-type theorems and a relative bound on the eigenvalues of regular
matrix pairs. We apply these bounds to the QEP and derive relative
perturbation bounds for eigenvectors and eigenvalues in
Section~\ref{Sec4}. Numerical examples are
given in Section~\ref{Sec5} to illustrate our bounds. Finally, some conclusions are summarized
in Section~\ref{Sec6}.

\medskip

\textbf{Notation.} Throughout this paper, we use $\|\cdot\|_2$,
$\|\cdot\|_F$ and $\|\cdot\|_{\UI}$ to denote the spectral matrix
norm, the Frobenius norm and any unitarily invariant matrix norm,
respectively.
We use $\kappa_2(\cdot)$ to denote the spectral condition number.
$I_n$ is the identity matrix of size $n$, and sometimes the subscript is omitted when the size is clear from the context.
If $A$ is a Hermitian positive (semi-)definite matrix, we will write
$A\succ0$ ($A\succeq 0$), and similarly $A\prec0$ ($A\preceq0$) for a
  Hermitian negative (semi-)definite $A$. $A^+$ is the Moore-Penrose pseudo-inverse
of $A$.
We use the standard MATLAB\textsuperscript{\textregistered} notation $A(:,i)$ for the $i$-th
column of the matrix $A$.
Also, we denote by $F_m$ and $G_m$, respectively,  the $m\times m$ matrices given by
\[
F_m =
\begin{bmatrix}
	  &         &         & 1 \\
	  &         & \iddots &   \\
	  & \iddots &         &   \\
	1 &         &         &
\end{bmatrix}
,
G_m =
\begin{bmatrix}
	  &         & 1       & 0 \\
	  & \iddots & \iddots &   \\
	1 & \iddots &         &   \\
	0 &         &         &
\end{bmatrix}
.
\]
Then the Jordan block of size $m$ corresponding to the eigenvalue $\lambda$ is
\[
J_m(\lambda) := \lambda I_m + G_mF_m=
\begin{bmatrix}
	\lambda & 1      &        &         \\
	        & \ddots & \ddots &         \\
	        &	     & \ddots & 1       \\
	        &	     &        & \lambda
\end{bmatrix}
.
\]


A matrix pair $A -\lambda B$ is sometimes denoted by $(A,B)$.
If $A_1,B_1$ have the same size and so do $A_2,B_2$, then the matrix pair $
\begin{bmatrix}
	A_1 & \\ & A_2
\end{bmatrix}
- \lambda
\begin{bmatrix}
	B_1 & \\ & B_2
\end{bmatrix}
$ is also written as
$ (A_1 -\lambda  B_1)\oplus (A_2 -\lambda  B_2) $ or $ (A_1,B_1)\oplus(A_2,B_2)$.
For any matrices $W,V$ of appropriate sizes, $W^H(A,B)V$ is used to denote the pair $(W^HAV,W^HBV)$.

\section{Canonical forms for matrix pairs}
As described in the introduction, we will derive perturbation bounds for regular Hermitian matrix pairs $(A,B)$, and we will use those bounds to derive bounds for the QEP. Thus, in this section, we recall a canonical form of Hermitian matrix pairs in Lemma~\ref{lm:canonical-form}, which will play a fundamental role in the discussion of eigenspaces.
\begin{lemma}[{\cite[Theorem~5.10.1]{golr:05}, \cite[Theorem~6.1]{laro:05}}]
	Every Hermitian matrix pair $(A,B)$ is congruent to a Hermitian matrix pair of the form
	\begin{multline}
			(0,0) \oplus \bigoplus_{j=1}^p (G_{2\varepsilon_j+1},\begin{bmatrix} &&F_{\varepsilon_j}\\ &0&\\F_{\varepsilon_j}&&\end{bmatrix})
				\\
				\oplus \bigoplus_{j=1}^r \delta_j(F_{k_j},G_{k_j})
				\oplus \bigoplus_{j=1}^q \eta_j(\alpha_jF_{\ell_j}+G_{\ell_j},F_{\ell_j})
				\\
				\oplus \bigoplus_{j=1}^s (\begin{bmatrix} &\beta_jF_{m_j}+G_{m_j}\\ \ol{\beta_j}F_{m_j}+G_{m_j}&\end{bmatrix},F_{2m_j})
					.
		\label{eq:lm:canonical-form}
	\end{multline}
	Here, $\varepsilon_1\le\dots\le\varepsilon_p$ and
        $k_1\le\dots\le k_p$, $l_1\le \dots \le l_q$ and $m_1\le \dots\le m_s$ are positive integers,  $\alpha_j$ are
        real numbers,  $\beta_j$ are complex non-real numbers, and
        $\delta_j,\eta_j$ are signs ($+1$ or $-1$).
	The form is uniquely determined by $(A,B)$ up to a combination
        of permutations of the following types of blocks:
	\begin{description}
		\item[T0.] $ (0,0)$;
		\item[T1.] $ (G_{2\varepsilon_j+1},\begin{bmatrix}  &&F_{\varepsilon_j}\\ &0&\\F_{\varepsilon_j}&&\end{bmatrix}) , \quad j=1,\dots, p$;
		\item[T2.] $ \delta_j(F_{k_j},G_{k_j}), \quad j=1,\dots, r$;
		\item[T3.] $ \eta_j(\alpha_jF_{\ell_j}+G_{\ell_j},F_{\ell_j}), \quad j=1,\dots, q$;
		\item[T4.] $ (\begin{bmatrix}
                    &\beta_jF_{m_j}+G_{m_j}\\
                    \ol{\beta_j}F_{m_j}+G_{m_j}&\end{bmatrix},F_{2m_j}),
                  \quad j=1,\dots, s$, with possible replacement of
                  $\beta_j$ by $\ol{\beta_j}$.
	\end{description}
	\label{lm:canonical-form}
\end{lemma}
Specifically, if $(A,B)$ is a regular pair, then its canonical form
only contains blocks of type \typ[T]{2}, \typ[T]{3}, \typ[T]{4}.

To indicate the similarity of the asymptotic behavior between these
blocks and the corresponding eigenvalue, in the following we apply a congruence transformation:
\begin{enumerate}
	\item
		For blocks of type \typ[T]{3} with $\alpha_j\ne0$,
		write $$T=\diag(|\alpha_j|^{\frac{\ell_j-1}{2}},|\alpha_j|^{\frac{\ell_j-1}{2}-1},\dots,|\alpha_j|^{-\frac{\ell_j-1}{2}}).$$
		Then $T^HF_{\ell_j}T=F_{\ell_j}$ and $T^HG_{\ell_j}T=|\alpha_j|G_{\ell_j}$, so the new block pair is of type $ \eta_j(\alpha_jF_{\ell_j}+|\alpha_j|G_{\ell_j},F_{\ell_j})$.
	\item
		For blocks of type \typ[T]{4},
		write $$T=\diag(|\beta_j|^{\frac{m_j-1}{2}},|\beta_j|^{\frac{m_j-1}{2}-1},\dots,|\beta_j|^{-\frac{m_j-1}{2}}).$$
		Then $T^HF_{m_j}T=F_{m_j}$ and $T^HG_{m_j}T=|\beta_j|G_{m_j}$. Since
		\[
			\begin{bmatrix} T&\\ &T
			\end{bmatrix}^H \begin{bmatrix} &S\\ S^H&
			\end{bmatrix} \begin{bmatrix} T&\\ &T
			\end{bmatrix} = \begin{bmatrix} &T^HST\\ T^HS^HT&
			\end{bmatrix},
		\]
		the new block pair is of type
		\[
			(\begin{bmatrix} &\beta_jF_{m_j}+|\beta_j|G_{m_j}\\ \ol{\beta_j}F_{m_j}+|\ol{\beta_j}|G_{m_j}&\end{bmatrix},F_{2m_j}).
		\]
\end{enumerate}

From this, we obtain the following variant of the canonical form:
\begin{lemma}
	Every regular Hermitian matrix pair $(A,B)$ is congruent to a Hermitian matrix pair of the form
	\begin{multline}
		\bigoplus_{j=1}^r \delta_j(F_{k_j},G_{k_j})
		\oplus \bigoplus_{j=1}^{q'} \eta'_j(G_{\ell'_j},F_{\ell'_j})
		\oplus \bigoplus_{j=1}^q \eta_j(\alpha_jF_{\ell_j}+|\alpha_j|G_{\ell_j},F_{\ell_j})
		\\
		\oplus \bigoplus_{j=1}^s (\begin{bmatrix} &\beta_jF_{m_j}+|\beta_j|G_{m_j}\\ \ol{\beta_j}F_{m_j}+|\ol{\beta_j}|G_{m_j}&\end{bmatrix},F_{2m_j})
		.
		\label{eq:lm:modified-canonical-form}
	\end{multline}
	Here, 
	 $k_j$, $l_j$ and $m_j$ are positive integers,  $\alpha_j$ are nonzero real numbers,
        $\beta_j$ are complex non-real numbers, and the
        $\delta_j,\eta'_j,\eta_j$ are signs ($+1$ or $-1$).
	This form is uniquely determined by $(A,B)$ up to a combination
        of permutations of the following types of blocks:
	\begin{description}
		\item[R1.] $ \delta_j(F_{k_j},G_{k_j}), \quad j=1,\dots, r$;
		\item[R2.] $ \eta'_j(G_{\ell'_j},F_{\ell'_j}), \quad j=1,\dots, q'$;
		\item[R3.] $ \eta_j(\alpha_jF_{\ell_j}+|\alpha_j|G_{\ell_j},F_{\ell_j}), \quad j=1,\dots, q$;
		\item[R4.] $ (\begin{bmatrix} &\beta_jF_{m_j}+|\beta_j|G_{m_j}\\ \ol{\beta_j}F_{m_j}+|\ol{\beta_j}|G_{m_j}&\end{bmatrix},F_{2m_j}), \quad j=1,\dots, s$, with possible replacement of $\beta_j$ by $\ol{\beta_j}$.
	\end{description}
	\label{lm:modified-canonical-form}
\end{lemma}

\section{Relative perturbation bound for a Hermitian matrix pair}\label{Sec3}

In this section, first we will see that the bound on the angle between the
eigenspaces of a regular Hermitian matrix pair is related to a bound
on the solutions to \emph{structured Sylvester equations}
$SX-XS'=T$ with coefficients $S,S',T$ of special structure.
Our approach more or less follows the lines of Davis and Kahan
\cite{Davis1970} and is very closely related to the work of Li
\cite{li:99,li:03}.

\subsection{Bound on the solution of the structured Sylvester equation}

Besides the references given above, structured Sylvester equations and
their connection to the $\sin \Theta$-type theorems can be found in, e.g., \cite{Li1999a}.
According to our needs, we derive Lemma~\ref{eq:lm:structured-Sylvester}.
\begin{lemma}
	Given two matrices $M,N$ of the same size.
	Suppose $(\Lambda,\Omega)$ and $(\Lambda',\Omega')$ are block
        pairs of the type \typ{2}, \typ{3}, \typ{4} in
        \eqref{eq:lm:modified-canonical-form}.
	Write
	$ \lambda(\Lambda,\Omega)=\{\underbrace{\lambda,\dots,\lambda}_{n}\} \text{\;or\;} \{\underbrace{\lambda,\ol{\lambda},\dots,\lambda,\ol{\lambda}}_{2n}\}$
	and
	$ \lambda(\Lambda',\Omega')=\{\underbrace{\lambda',\dots,\lambda'}_{n'}\} \text{\;or\;} \{\underbrace{\lambda',\ol{\lambda'},\dots,\lambda',\ol{\lambda'}}_{2n'}\}$.
	For any $\alpha>0,m\in\mathbb{N}$, define
	\[
	\varphi_-(\alpha,m):=\alpha^{1-m}\frac{1-\alpha^{m}}{1-\alpha},\quad
	\varphi_+(\alpha,m):=\alpha^{-m}\frac{1+\alpha-2\alpha^{m}}{1-\alpha}.
	\]
	If $ \lambda(\Lambda,\Omega)\cap\lambda(\Lambda',\Omega')=\emptyset$,
	then each of the equations
	\begin{subequations}		\label{eq:lm:structured-Sylvester}
		\begin{gather}
			(\Omega'\Lambda')^HY - Y\Omega\Lambda =
                        -(\Omega'\Lambda')^HM + N\Omega\Lambda, \label{eq:lm:structured-Sylvester:cont}
			\\
			Y - (F_{n'}G_{n'})^HY\Omega\Lambda =  - M +
                        (F_{n'}G_{n'})^HN\Omega\Lambda, \label{eq:lm:structured-Sylvester:disc}
		\end{gather}
	\end{subequations}%
	has a unique solution $Y$ whose size is the same as that of $M,N$.
	For \eqref{eq:lm:structured-Sylvester:disc} we define $\lambda'=\infty$.
	Then, for both \eqref{eq:lm:structured-Sylvester:cont} and \eqref{eq:lm:structured-Sylvester:disc},
	the solution $Y$ is bounded by
	\[
	\|Y\|_F\le
	\phi_1(\lambda,\lambda',n,n')\|M\|_F+\phi_2(\lambda,\lambda',n,n')\|N\|_F,
	\]
	where
        \begin{equation}\label{def:alpha}
	\begin{aligned}
		&\phi_1(\lambda,\lambda',n,n') &&=
                \tbinom{n'+n-1}{n}\varphi_-(\gamma,n)\varphi_+(\gamma',n'),&\quad\\
		&\phi_2(\lambda,\lambda',n,n') &&=
                \tbinom{n'+n-1}{n'}\varphi_-(\gamma',n')\varphi_+(\gamma,n),\\
		&\phi_1(0,\lambda',n,n') &&=
                \tbinom{n'+n-1}{n}|\lambda'|^{-1}\varphi_-(|\lambda'|,n-1)+1,&  \\
		&\phi_2(0,\lambda',n,n') &&=
                \tbinom{n'+n-1}{n}|\lambda'|^{-1}\varphi_-(|\lambda'|,n-1), \\
		&\phi_1(\lambda,0,n,n') &&=
                \tbinom{n'+n-1}{n'}|\lambda|^{-1}\varphi_-(|\lambda|,n'-1),&\quad \\
		&\phi_2(\lambda,0,n,n') &&=
                \tbinom{n'+n-1}{n'}|\lambda|^{-1}\varphi_-(|\lambda|,n'-1)+1,\\
		&\phi_1(\lambda,\infty,n,n') &&=
                2|\lambda|\varphi_-(\frac{1}{2|\lambda|},n'-1)+1,& \\
		&\phi_2(\lambda,\infty,n,n') &&=
                2|\lambda|\varphi_-(\frac{1}{2|\lambda|},n'-1), \\
		&\phi_1(0,\infty,n,n') &&= \min\{n,n'\},&\quad \\
		&\phi_2(0,\infty,n,n') &&= \min\{n,n'\}-1,
	\end{aligned}
        \end{equation}
		in which
 $\lambda,\lambda'\notin\{0,\infty\}$, and
	$\gamma := \min\{\left|\frac{\lambda'-\lambda}{\lambda}\right|,\left|\frac{\ol{\lambda'}-\lambda}{\lambda}\right|\},\gamma' := \min\{\left|\frac{\lambda'-\lambda}{\lambda'}\right|,\left|\frac{\ol{\lambda'}-\lambda}{\lambda'}\right|\}$.
	\label{lm:structured-Sylvester}
\end{lemma}
The proof of Lemma~\ref{lm:structured-Sylvester} is long, quite
technical, and provides no further insight, so we defer it to \ref{sec:proof-lm:structured-Sylvester}.
\bigskip

\subsection{The main result}

In the following, we will present the relative perturbation theory for
Hermitian matrix pairs, using the above results. To derive our bounds we need to specify what assumptions we will use within this paper.\\
We consider two Hermitian matrix pairs $(A,B)$ and $(\wtd A,\wtd B)$,
and two nonsingular matrices $X=\begin{bmatrix} X_1 &
  X_2 \end{bmatrix}$ and $\wtd X=\begin{bmatrix} \wtd X_1 & \wtd
  X_2 \end{bmatrix}$ for which the following assumptions hold:
\begin{enumerate}[({A}1)]
	\item Both matrix pairs are regular.
	\item $X,\wtd X$ are the matrices to canonicalize the unperturbed and perturbed matrix pairs with special partitions. In detail, they satisfy the following statements:
		\begin{enumerate}
			\item $X^H(A,B)X=(\Lambda,\Omega) = (\Lambda_1,\Omega_1)\oplus(\Lambda_2,\Omega_2)$, or equivalently, $X_i^HAX_i=\Lambda_i, X_i^HBX_i=\Omega_i, i=1,2$ and $X_1^HAX_2=X_1^HBX_2=0$, where $(\Lambda_i,\Omega_i), i=1,2$ are of the form \eqref{eq:lm:modified-canonical-form}, namely
		\begin{equation}\label{eq:blockdetail}
			 (\Lambda_i,\Omega_i)=\bigoplus_{j_i=1}^{m_i}(\Lambda_{i,j_i},\Omega_{i,j_i}), i=1,2,
		\end{equation}
		where $(\Lambda_{i,j_i},\Omega_{i,j_i})$ are blocks of size $n_{i,j_i}$ as in \eqref{eq:lm:modified-canonical-form} with corresponding eigenvalue $\lambda_{i,j_i}$.
			\item Similarly,  $\wtd X^H(\wtd A,\wtd B)\wtd X=(\wtd \Lambda,\wtd \Omega) = (\wtd \Lambda_1,\wtd \Omega_1)\oplus(\wtd \Lambda_2,\wtd \Omega_2)$, or equivalently, $\wtd X_i^H\wtd A\wtd X_i=\wtd \Lambda_i, \wtd X_i^H\wtd B\wtd X_i=\wtd \Omega_i, i=1,2$ and $\wtd X_1^H\wtd A\wtd X_2=\wtd X_1^H\wtd B\wtd X_2=0$,
				where $(\wtd \Lambda_i,\wtd \Omega_i), i=1,2$ are of the form \eqref{eq:lm:modified-canonical-form}, namely
		\begin{equation}\label{eq:blockdetail:wtd}
			(\wtd \Lambda_i,\wtd \Omega_i)=\bigoplus_{\wtd j_i=1}^{\wtd m_i}(\wtd \Lambda_{i,\wtd j_i},\wtd \Omega_{i,\wtd j_i}), i=1,2,
		\end{equation}
		where $(\wtd \Lambda_{i,\wtd j_i},\wtd \Omega_{i,\wtd j_i})$ are blocks of size $\wtd n_{i,\wtd j_i}$ as in \eqref{eq:lm:modified-canonical-form} with corresponding eigenvalue $\wtd \lambda_{i,j_i}$.
	\item $\Lambda_i,\Omega_i,\wtd \Lambda_i, \wtd \Omega_i$ are of the same size for $i=1,2$, which means that $X_i,\wtd X_i$ are of the same size for $i=1,2$, too. Note that this only means $\sum_{j_i=1}^{m_i} n_{i,j_i}=\sum_{\wtd j_i=1}^{\wtd m_i}\wtd n_{i,\wtd j_i}$ rather than any further assumption on $m_i, n_{i,j_i},\wtd m_i, \wtd n_{i,\wtd j_i}$.
		\end{enumerate}
	\item It holds that
		\begin{align*}
			\lambda(\Lambda_1,\Omega_1)\cap\lambda(\Lambda_2,\Omega_2)=\emptyset,
			&\quad
			\lambda(\wtd \Lambda_1,\wtd \Omega_1)\cap\lambda(\wtd \Lambda_2,\wtd \Omega_2)=\emptyset,
			\\
			\lambda(\Lambda_1,\Omega_1)\cap\lambda(\wtd \Lambda_2,\wtd \Omega_2)=\emptyset,
			&\quad
			\lambda(\Lambda_2,\Omega_2)\cap\lambda(\wtd \Lambda_1,\wtd \Omega_1)=\emptyset.
		\end{align*}
	\item It holds that
		\[
			\infty\notin\lambda(\Lambda_1,\Omega_1),
			\quad
			0\notin\lambda(\Lambda_2,\Omega_2),
			\quad
			\infty\notin\lambda(\wtd \Lambda_1,\wtd \Omega_1),
			\quad
			0\notin\lambda(\wtd \Lambda_2,\wtd \Omega_2),
		\]
		which implies $\Omega_1,\wtd \Omega_1,\Lambda_2,\wtd \Lambda_2$ are nonsingular.
\end{enumerate}
\begin{remark}\label{rk:assumption}
	Let us say something more to illustrate the practical meaning of such complicated assumptions.
	Clearly, (A1) is necessary.

	By (A2) we partition the whole space into two (eigen-)subspaces $X_1,X_2$, such that $X_1$ makes $(A,B)$ to behave as $(\Lambda_1,\Omega_1)$ and $X_2$ makes $(A,B)$ to behave as $(\Lambda_2,\Omega_2)$.
	Without loss of generality, we may call the former one ``focused'' and the latter one ``ignored''.
	Similarly to this, for perturbed case we partition the whole space into two (eigen-)subspaces $\wtd X_1,\wtd X_2$ such that one makes $(\wtd A,\wtd B)$ to behave as $(\wtd \Lambda_1,\wtd \Omega_1)$ and other makes $(\wtd A,\wtd B)$ to behave as $(\wtd \Lambda_2,\wtd \Omega_2)$.
	We hope to observe the effect of the perturbation, so $X_1,\wtd X_1$ should be related and very close to each other if the perturbation is small enough.
	As a result $\Lambda_1,\wtd\Lambda_1$ are related, and so do $\Omega_1,\wtd \Omega_1$, and also holds for the ignored subspaces.

	(A3) is a mathematical form of the practical requirement: the eigenvalues in the focused subspaces and the ignored subspaces may not be the same or close, otherwise we cannot distinguish the effects from the focused and ignored subspaces.

	(A4) is a technical assumption: we focus the possible eigenvalue $0$ and ignore the possible eigenvalue $\infty$.
	Under these assumption, the purpose of the perturbation analysis is to find a quantitative description how close the eigenvalues and eigenspaces of the unperturbed and perturbed matrix pairs are.

	Here is another implicit technical requirement that we will use later.
	If a complex eigenvalue is focused, then by the blocks of type $\typ{4}$, its conjugate, whish is also an eigenvalue, is focused, too.
	It is quite possible that the related eigenvalue of the perturbed matrix pair is also complex.
	We would have to measure how close two conjugate-closed complex number pairs are, and a natural way is to measure the complex unperturbed and perturbed eigenvalues with non-negative (or non-positive) imaginary parts.
\end{remark}
Now we can state the central theorem of this paper, namely Theorem~\ref{thm:main_thm}.
\begin{theorem}\label{thm:main_thm}
	Let $(A,B)$ and $(\wtd A, \wtd B)=(A+\delta A, B+\delta B)$ be regular
        Hermitian matrix pairs and $X$ and $\wtd X$ be nonsingular matrices.
	If the assumptions (A1)--(A4) hold for the perturbed and unperturbed
        matrix pairs, then we have
	\begin{equation}
\begin{aligned}
		\|\sin\Theta(\subspan(X_1),\subspan(\wtd X_1))\|_F
		\le
		\phi_1^{\max} &\kappa_2(X)\kappa_2(\wtd X)\|A^{\pinv}\delta A\|_F\\ &+\phi_2^{\max} \|X\|_2^2\|\wtd X^{-1}\|_2^2\|\wtd B^{\pinv}\delta B\|_F,
\end{aligned}
		\label{eq:thm:subspace-relative-bound-(A,B):regular:linear}
	\end{equation}
where
\begin{equation}\label{def:phi_12^m}
	\phi_1^{\max}:=\max_{j_1,j_2}\phi_1(\wtd\lambda_{1,j_1},\lambda_{2,j_2},\wtd n_{1,j_1},n_{2,j_2}), \quad
	\phi_2^{\max}:=\max_{j_1,j_2}\phi_2(\wtd\lambda_{1,j_1},\lambda_{2,j_2},\wtd n_{1,j_1},n_{2,j_2}),
\end{equation}
	and $\phi_1(\cdot,\cdot,\cdot,\cdot),\phi_2(\cdot,\cdot,\cdot,\cdot)$ 
	are defined in \eqref{def:alpha}.
	Moreover, if $B$ is nonsingular,
	\begin{equation}
		\|\sin\Theta(\subspan(X_1),\subspan(\wtd X_1))\|_F
		\le \kappa_2(X)\kappa_2(\wtd X)\left(\phi_1^{\max} \|A^{\pinv}\delta A\|_F+\phi_2^{\max} \|B^{-1}\delta B\|_F\right).
		\label{eq:thm:subspace-relative-bound-(A,B):regular:Bnonsingular:linear}
	\end{equation}
	\label{thm:subspace-relative-bound-(A,B):regular}
\end{theorem}
\begin{proof}
	First, similarly to the discussion in \cite[Lemma~2]{TruMio}, using $QR$ decomposition of the matrices $BX_2$ and $\wtd X_1$ and the fact that $\|\sin \Theta(X_1, \wtd X_1)\|=\|Q_2^*\wtd Q_1\|$, where $Q_2$ and $\wtd Q_1$ form the orthonormal basis for $\subspan(X_2)$ and $\subspan(\wtd X_1)$, we observe that
	\[
	\|\sin\Theta(\subspan(X_1),\subspan(\wtd X_1))\|_F\le
	\|(BX_2)^{\pinv}\|_2\|\wtd X_1^{\pinv}\|_2\| \|X_2^HB\wtd X_1\|_F.
	\]
	From $BX_2 = X^{-H}\begin{bmatrix}0\\ \Omega_2
	\end{bmatrix}$, we get $\|BX_2\|_2\le\|X^{-1}\|_2$ and
        $\|(BX_2)^{\pinv}\|_2=\|\begin{bmatrix} 0 & \Omega_2^H
	\end{bmatrix}^{\pinv}X\|_2\le\|X\|_2$.
	Similarly, $\|\wtd B\wtd X_1\|_2\le\|\wtd X^{-1}\|_2$.

	Then, to bound $\|\sin\Theta(\subspan(X_1),\subspan(\wtd X_1))\|_F$, we
        need to estimate $\|X_2^HB\wtd X_1\|_F$.
	Using \eqref{eq:blockdetail}, define $L, R$ by:
	\[
	L = L_1\oplus L_2 = \bigoplus_{i=1}^2\bigoplus_{j_i=1}^{m_i}L_{i,j_i},\quad
	R = R_1\oplus R_2 = \bigoplus_{i=1}^2\bigoplus_{j_i=1}^{m_i}R_{i,j_i},
	\]
	where $L_{i,j_i} = \Lambda_{i,j_i}\Omega_{i,j_i},R_{i,j_i} = I$ if the corresponding block pair is of type \typ{1} or $L_{i,j_i} = I, R_{i,j_i} = \Omega_{i,j_i}\Lambda_{i,j_i}$ if not.
	Here the size of $I$ is determined under the condition that $L_{i,j_i}$ and $R_{i,j_i}$ are of the same size.
	Then, it is easy to check that $L$ and $R$ are lower triangular matrices since $ \Lambda_{i,j_i}$ and $\Omega_{i,j_i}$ are matrices from the pairs of the form \eqref{eq:lm:modified-canonical-form}. Also $L_1=I$ by (A4).
	Note that if $L_{i,j_i} = \Lambda_{i,j_i}\Omega_{i,j_i},R_{i,j_i} = I$, when the corresponding block pair is of type \typ{1}, then it is easy to see that $\Lambda_{i,j_i}\Lambda_{i,j_i}=I$ and we have $\Lambda L=\Omega=\Omega R$. Also, in second case when $L_{i,j_i} = I, R_{i,j_i} = \Omega_{i,j_i}\Lambda_{i,j_i}$, it will be that $\Omega_{i,j_i}\Omega_{i,j_i}=I$ and $\Omega R=\Lambda=\Lambda L$. Thus, we can conclude that
	\[
	X^HAXL = \Lambda L = \Omega R = X^HBXR
	\]
	and then 
	\begin{equation}
		AXL=BXR,\quad \text{and similarly}\quad \wtd A\wtd X\wtd L = \wtd B\wtd X\wtd R.
		\label{eq:eigen-structure}
	\end{equation}
	Also, $\wtd L$ and $\wtd R$ are lower triangular matrices and $\wtd L_1=I$.
	Then
	\begin{align*}
		R_2^HX_2^HB\wtd X_1\wtd L_1 &- L_2^HX_2^HB\wtd X_1\wtd R_1\\
		&= R_2^HX_2^HB\wtd X_1\wtd L_1 - L_2^HX_2^H\wtd B\wtd X_1\wtd R_1 + L_2^HX_2^H\delta B\wtd X_1\wtd R_1
		\\&= L_2^HX_2^HA\wtd X_1\wtd L_1 - L_2^HX_2^H\wtd A\wtd X_1\wtd L_1 + L_2^HX_2^H\delta B\wtd X_1\wtd R_1
		\\&=  - L_2^HX_2^H\delta A\wtd X_1\wtd L_1 + L_2^HX_2^H\delta B\wtd X_1\wtd R_1.
	\end{align*}
	Note that $AX_2 = X^{-H}\begin{bmatrix}0\\ \Lambda_2
	\end{bmatrix}$.
	Since $\Lambda_2$ is nonsingular, $AX_2$ has full column rank, which
        implies that $\subspan(X_2)\cap\nullspace(A)=\emptyset$, where $\nullspace(A)$
        is the nullspace of $A$.
	Note that $I-A^{\pinv}A$ is the orthogonal projector onto $\nullspace(A)$.
	Thus, $A^{\pinv}AX_2=X_2$, whence
	\begin{equation}
		\begin{aligned}
			R_2^HX_2^HB\wtd X_1\wtd L_1 - L_2^HX_2^HB\wtd X_1\wtd R_1
			&=  - L_2^HX_2^HAA^{\pinv}\delta A\wtd X_1\wtd L_1 +
                        L_2^HX_2^H\delta B\wtd X_1\wtd R_1
			\\&=  - R_2^HX_2^HBA^{\pinv}\delta A\wtd X_1\wtd L_1 +
                        L_2^HX_2^H\delta B\wtd X_1\wtd R_1.
		\end{aligned}
		\label{eq:structured-Sylvester:to-be-estimated}
	\end{equation}
	To estimate the norm of the solution $X_2^HB\wtd X_1$ to the structured Sylvester equation
        \eqref{eq:structured-Sylvester:to-be-estimated},
	we will use the fact that the operator $Y\mapsto R_2^HY\wtd L_1-L_2^HY\wtd R_1$ is linear.
	Hence, if $Y^{(1)}$ is the solution to
	\begin{equation}\label{eq:structured-Sylvester:to-be-estimated:1}
			R_2^HY\wtd L_1 - L_2^HY\wtd R_1
			=  - R_2^HX_2^HBA^{\pinv}\delta A\wtd X_1\wtd L_1,
	\end{equation}
	and $Y^{(2)}$ is the solution to
	\begin{equation}\label{eq:structured-Sylvester:to-be-estimated:2}
			R_2^HY\wtd L_1 - L_2^HY\wtd R_1
			= L_2^HX_2^H\delta B\wtd X_1\wtd R_1,
		\end{equation}
	then $X_2^HB\wtd X_1=Y^{(1)}+Y^{(2)}$ is the solution to \eqref{eq:structured-Sylvester:to-be-estimated}.
	Thus estimating the norms of the two $Y^{(i)}$ and then using
        $\|X_2^HB\wtd X_1\|\le\|Y^{(1)}\|+\|Y^{(2)}\|$ (for any norm) will yield the result.

		For this, first we consider $Y^{(1)}$ and \eqref{eq:structured-Sylvester:to-be-estimated:1}, related to $\delta B=0$ in
        \eqref{eq:structured-Sylvester:to-be-estimated}, i.e.,
	for $j_1=1,\dots,m_1$ and $j_2=1,\dots,m_2$:
	\[
	R_{2,j_2}^HY^{(1)}_{j_2,j_1}\wtd L_{1,j_1} - L_{2,j_2}^HY^{(1)}_{j_2,j_1}\wtd R_{1,j_1}
	=  - R_{2,j_2}^HX_{2,j_2}^HBA^{\pinv}\delta A\wtd X_{1,j_1}\wtd L_{1,j_1}
	.
	\]
	Noticing $\wtd L_1=I$,
	this equation in $Y^{(1)}_{j_2,j_1}$ is of the form \eqref{eq:lm:structured-Sylvester}, i.e., of the form \eqref{eq:lm:structured-Sylvester:cont} or \eqref{eq:lm:structured-Sylvester:disc} depending on the type of corresponding block pairs in \eqref{eq:blockdetail}.
	By Lemma~\ref{lm:structured-Sylvester},
	\[
	\|Y^{(1)}_{j_2,j_1}\|_F \le \phi_1(\wtd \lambda_{1,j_1},\lambda_{2,j_2},\wtd n_{1,j_1},n_{2,j_2})\|X_{2,j_2}^HBA^{\pinv}\delta A\wtd X_{1,j_1}\|_F.
	\]
	Then
	\begin{align*}
		\|Y^{(1)}\|_F
		&= \sqrt{\sum_{j_1=1}^{m_1}\sum_{j_2=1}^{m_2}\|Y^{(1)}_{j_2,j_1}\|_F^2}
		\\&\le \max_{j_1,j_2}\phi_1(\wtd \lambda_{1,j_1},\lambda_{2,j_2},\wtd n_{1,j_1},n_{2,j_2})\sqrt{\sum_{j_1=1}^{m_1}\sum_{j_2=1}^{m_2}\|X_{2,j_2}^HBA^{\pinv}\delta A\wtd X_{1,j_1}\|_F^2}
		\\&\le \max_{j_1,j_2}\phi_1(\wtd \lambda_{1,j_1},\lambda_{2,j_2},\wtd n_{1,j_1},n_{2,j_2})\|X_2^HBA^{\pinv}\delta A\wtd X_1\|_F
		\\&\le \max_{j_1,j_2}\phi_1(\wtd \lambda_{1,j_1},\lambda_{2,j_2},\wtd n_{1,j_1},n_{2,j_2})\|BX_2\|_2\|\wtd X_1\|_2\|A^{\pinv}\delta A\|_F
		\\&\le \max_{j_1,j_2}\phi_1(\wtd \lambda_{1,j_1},\lambda_{2,j_2},\wtd n_{1,j_1},n_{2,j_2})\|X^{-1}\|_2\|\wtd X\|_2\|A^{\pinv}\delta A\|_F.
	\end{align*}
	For $Y^{(2)}$, which is solution of \eqref{eq:structured-Sylvester:to-be-estimated:2}, related to $\delta A=0$ in
        \eqref{eq:structured-Sylvester:to-be-estimated}, we obtain
	\[
	\|Y^{(2)}\|_F
	\le \max_{j_1,j_2}\phi_2(\wtd \lambda_{1,j_1},\lambda_{2,j_2},\wtd n_{1,j_1},n_{2,j_2})\|X_2^H\delta B\wtd X_1\|_F.
	\]

	In case where $B$ is nonsingular,
	\[
	\|Y^{(2)}\|_F
	\le \max_{j_1,j_2}\phi_2(\wtd \lambda_{1,j_1},\lambda_{2,j_2},\wtd n_{1,j_1},n_{2,j_2})\|X^{-1}\|_2\|\wtd X\|_2\|B^{-1}\delta B\|_F.
	\]
	Thus \eqref{eq:thm:subspace-relative-bound-(A,B):regular:Bnonsingular:linear} holds.

	For the general case,
	since $\wtd \Omega_1$ is nonsingular, $\wtd B^{\pinv}\wtd B\wtd X_1=\wtd X_1$.
	Then
	\[
	\|Y^{(2)}\|_F
	\le \max_{j_1,j_2}\phi_2(\wtd \lambda_{1,j_1},\lambda_{2,j_2},\wtd n_{1,j_1},n_{2,j_2})\|X\|_2\|\wtd X^{-1}\|_2\|\wtd B^{\pinv}\delta B\|_F,
	\]
	so \eqref{eq:thm:subspace-relative-bound-(A,B):regular:linear} holds.
\end{proof}
Let us illustrate the quality of bounds \eqref{eq:thm:subspace-relative-bound-(A,B):regular:linear} and \eqref{eq:thm:subspace-relative-bound-(A,B):regular:Bnonsingular:linear} with the following small example.
\begin{example}
Consider a regular matrix pair $(A,B)$, where
\begin{eqnarray*} A=\left[\begin{matrix} 0 &8 &\varrho &2\\ 8 & 0 & 2& 0\\ \varrho & 2 & 0 &8 \\ 2& 0& 8&0\end{matrix}\right] \quad  \textrm{and} \quad B=\left[\begin{matrix} 0 & 0&0&1\\ 0 & 0 & 1& 0\\ 0 & 1 & 0 &0 \\ 1& 0& 0&0\end{matrix}\right],\end{eqnarray*}
where $|\varrho| \le 1$. The corresponding perturbed pair $(\wtd A, \wtd B)$ is given as $\wtd A= A+ \xi \delta A$ and $\wtd B=B+\xi \delta B$, where $\xi=10^{-8}$ and $\delta A$, $\delta B$ are random Hermitian matrices with elements uniformly distributed in interval $(0,1)$.  We will analyze the quality of the bound \eqref{thm:subspace-relative-bound-(A,B):regular} for a different choices of the value $\varrho$ repeated a $100$ times. The maximal and the minimal values of the quantities which appear in the bound  \eqref{thm:subspace-relative-bound-(A,B):regular} for $100$ generated perturbations $(\delta A, \delta B)$ are given in the Table \ref{table}.

\begin{table}[!ht]
{\tiny
\begin{tabular}{|l|c|c|c|c|l|l|l|}
\hline
\multicolumn{1}{|c|}{$\varrho$}& \multicolumn{2}{c|}{Exact $\|\sin\Theta\|_F$ } &\multicolumn{2}{c|}{Bound \eqref{thm:subspace-relative-bound-(A,B):regular}}& \multicolumn{1}{c|}{$\kappa_2(X)$}& \multicolumn{2}{c|}{$\kappa_2(\wtd X)$}\\
\cline{2-8}
 & min & max & min & max &  & min & max\\
 \hline
 $0$& $5.4812\cdot 10^{-10}$&   $2.8099\cdot10^{-9}$&  $5.4311\cdot 10^{-8}$ & $ 2.0468\cdot 10^ {-6}$& $1$&$1.0766$&$16.770$ \\
 $10^{-14}$ & $1.5450\cdot 10^{-10}$& $2.7125\cdot 10^{-9}$ & $4.0485\cdot 10^{-7}$& $6.5515\cdot 10^{-6}$ & $9.7053$&$1.0928$&$13.687$\\
 $1$&  $6.2533 \cdot 10^{-10} $ & $ 3.0239\cdot 10^{-9}$ & $3.1436\cdot 10^{-4}$ & $8.6771 \cdot 10^{-3}$ & $1$ & $6.9469\cdot 10^{3}$ & $2.1380 \cdot 10^{5}$\\
 \hline \end{tabular}
}
\caption{\label{table} Comparison of the exact value for $\|\sin\Theta(\subspan(X_1),\subspan(\wtd X_1))\|_F$   and the bound \eqref{thm:subspace-relative-bound-(A,B):regular}.}
\end{table}
For the case when $\varrho=0$ eigenvalues $-6$ and $10$ are both semi-simple and in that case our bound gives the best result because the conditions of the subspace matrices $X$ and $\wtd X$ are small since in both cases subspaces are spanned by linear independent eigenvectors.\\


Note, that for $\varrho\neq 0$ the eigenvalues of the pair $(A,B)$ correspond with the following Jordan form $$J=\textrm{diag}\left\{\begin{bmatrix}-6&1\\0& -6\end{bmatrix}, \begin{bmatrix}10&1\\0 & 10\end{bmatrix}\right\}.$$
In that case,  if $\varrho=1$, in $75$ of $100$ cases our bound is of the size $10^{-4}$, which is acceptable
since the condition number the matrix $\wtd X$ vary between $7 \cdot 10^3$ and  $2.2 \cdot 10^5$. It is important to emphasize that in
the presence of Jordan form the bound \eqref{thm:subspace-relative-bound-(A,B):regular} is given for the subspaces spanned by the columns of $X_1=X(:,1:2)$ and $\wtd X_1=\wtd X(:,1:2)$. The columns of $X_1$  contain one eigenvector and one generalized eigenvector, respectively, which correspond with double unperturbed eigenvalue from the Jordan form, while the columns of $\wtd X_1$ contain two eigenvectors which correspond with the perturbed eigenvalues which are simple, since under the small perturbation perturbed pair lose the Jordan form. A similar example, but for a regular quadratic eigenvalue problem, will be given in the last section (Example \ref{eg:Jordan-block}).\\

An interesting situation occurs, for a small $\varrho$, for example $\varrho=10^{-14}$. In that case, the rounding error in calculation of the eigenvectors (for example using \textsc{ Matlab} function \verb"eig") for the matrix pair $(A,B)$ will give four linear independent eigenvectors in $X$ and $\wtd X$ and Jordan form will not be detected. In that case our bound is again acceptable, and   give us good information about size of the perturbation.


%

\end{example}

Also, a more general but weaker result can be given for any unitarily invariant
norm, as is shown in Theorem~\ref{thm:subspace-relative-bound-(A,B):Bnonsingular:UI}. 
\begin{theorem}
	Given a Hermitian matrix pair $(A,B)$ and its corresponding perturbed
        pair $(\wtd A,\wtd B)$ with $B,\wtd B$ nonsingular (which guarantees
        assumption (A1)). Under the assumptions (A2)--(A4),
	if there exist $\alpha\ge 0$ and $\delta>0$ such that
	\begin{gather*}
          \begin{aligned}
		& \|\Omega_2\Lambda_2\|_2\le \alpha,\quad &&\|(\wtd\Omega_1 \wtd \Lambda_1)^{-1}\|_2^{-1}\ge\alpha+\delta,
		\qquad \text{or}\\
		& \|(\Omega_2\Lambda_2)^{-1}\|_2^{-1}\ge\alpha+\delta,\quad &&
                \|\wtd\Omega_1 \wtd \Lambda_1\|_2\le \alpha,
           \end{aligned}
	\end{gather*}
	then, for any unitarily invariant norm $\|\cdot\|_{\UI}$ and $p,q$ where $p^{-1}+q^{-1}=1$,
	\begin{equation*}
		\|\sin\Theta(\subspan(X_1),\subspan(\wtd X_1))\|_{\UI}
		\le \mu\kappa_2(X)\kappa_2(\wtd X)\sqrt[q]{\|A^{\pinv}\delta A\|_{\UI}^q+ \|B^{-1}\delta B\|_{\UI}^q},
	\end{equation*}
	where $\mu=\frac{\delta}{\sqrt[p]{\alpha^p+(\alpha+\delta)^p}}$. Also,
	also
	\begin{equation*}
		\|\sin\Theta(\subspan(X_1),\subspan(\wtd X_1))\|_{\UI}
		\le \kappa_2(X)\kappa_2(\wtd X)\left(\frac{\|A^{\pinv}\delta A\|_{\UI}}{\frac{\delta}{\alpha}}+\frac{\|B^{-1}\delta B\|_{\UI}}{\frac{\delta}{\alpha+\delta}}\right).
	\end{equation*}

	\label{thm:subspace-relative-bound-(A,B):Bnonsingular:UI}
\end{theorem}
\begin{proof}
	The proof is similar to the proof of
        Theorem~\ref{thm:subspace-relative-bound-(A,B):regular}, except a
        modified version of \cite[Lemma~2.3]{li:99} is used to estimate
        \eqref{eq:structured-Sylvester:to-be-estimated} rather than
        Lemma~\ref{lm:structured-Sylvester}.
	The ``modified'' version is to get rid of the assumption
        ``$\Omega,\Gamma$ are Hermitian''. That assumption appears there for
        using \cite[Lemma~2.2]{li:99} to make it clear that
        \eqref{eq:structured-Sylvester:to-be-estimated} has a unique
        solution. But now this is guaranteed by
        Lemma~\ref{lm:structured-Sylvester}.
\end{proof}
\begin{remark}
	Theorem~\ref{thm:subspace-relative-bound-(A,B):Bnonsingular:UI} is not
        that useful in general. If there is no semi-simple eigenvalue,
        we have the bounds
	\[
	\|\Omega_2\Lambda_2\|_2\le 2\max_{i_2}|\lambda_{2,i_2}|,
	\quad
	\|(\wtd\Omega_1\wtd\Lambda_1)^{-1}\|_2\le \max_{\wtd i_1}n_{\wtd i_1}|\wtd\lambda_{1,\wtd i_1}^{-1}|.
	\]
	This means that the gap must be bigger than $(2-\frac{1}{2})\alpha+2\delta$,
        which is not the expected one and it would be $k\delta$ for some scalar
        $k$.
	This theorem is good only for the case that all eigenvalues are semi-simple.
\end{remark}

Besides, if only one eigenvalue is considered, we obtain the eigenvalue perturbation bounds in Theorem~\ref{thm:eigenvalue-relative-bound-(A,B):Bnonsingular}.
\begin{theorem}
	Let $(A,B)$ and $(\wtd A, \wtd B)=(A+\delta A, B+\delta B)$ be regular Hermitian matrix pairs where $B,\wtd B$ are nonsingular, and let $X$ and $\wtd X$ be nonsingular matrices.

	Suppose that $X=\begin{bmatrix}
		X_1&\dots&X_m
	\end{bmatrix}$ is the matrix to canonicalize the unperturbed matrix pair with special partitions, namely
	$X^H(A,B)X=(\Lambda,\Omega) = \bigoplus_{j=1}^{m_i}(\Lambda_{j},\Omega_{j})$, or equivalently, $X_j^HAX_j=\Lambda_j, X_j^HBX_j=\Omega_j, j=1,\dots,m$ and $X_j^HAX_{j'}=X_j^HBX_{j'}=0,j\ne j'$,
	where $(\Lambda_{j},\Omega_{j})$ are blocks of size $n_{j}$ if the corresponding eigenvalue is real or $2n_{j}$ if the corresponding eigenvalue is imaginary as in \eqref{eq:lm:modified-canonical-form} with corresponding eigenvalue $\lambda_{j}$.
	Analogously $\wtd X=\begin{bmatrix}
		\wtd X_1&\dots&\wtd X_{\wtd m}
	\end{bmatrix}$, and $\wtd \Lambda_{\wtd j},\wtd \Omega_{\wtd j}, \wtd n_{\wtd j}, \wtd\lambda_{\wtd j},\wtd m$ play similar roles for $(\wtd A,\wtd B)$.

	Suppose that $\lambda\in\lambda(\Lambda_j,\Omega_j),\wtd \lambda\in\lambda(\wtd\Lambda_{\wtd j},\wtd\Omega_{\wtd j})$,
	where $\lambda, \wtd \lambda$ are nonzero and lie in the upper half plane. 
	Let $\gamma = \left|\frac{\wtd\lambda-\lambda}{\lambda}\right|,\wtd\gamma = \left|\frac{\wtd\lambda-\lambda}{\wtd\lambda}\right|$.
	Then,
	\begin{enumerate}
		\item provided $\gamma\le\frac{5-\sqrt{17}}{4}$, $\wtd\gamma\le\frac{5-\sqrt{17}}{4}$,
			\[
				\frac{\gamma^{ n_j}\wtd\gamma^{\wtd  n_{\wtd j}}}{\sqrt{\gamma^2+\wtd\gamma^2}}
				\le 2\frac{\|X^{-1}\|_2\|\wtd X\|_2 }{\|X_j^HB\wtd X_{\wtd j}\|_F}\sqrt{\tbinom{\wtd n_{\wtd j}+ n_j-1}{ n_j}^2\|A^{\pinv}\delta A\|_F^2 + \tbinom{\wtd n_{\wtd j}+ n_j-1}{\wtd n_{\wtd j}}^2\|B^{-1}\delta B\|_F^2};
			\] 
		\item Specially, if $ n_j=\wtd  n_{\wtd j}=1$ and $|x^HB\wtd x|>\|X^{-1}\|_2\|\wtd X\|_2\|B^{-1}\delta B\|_{\UI}$, where $x,\wtd x$ are the corresponding eigenvectors of $\lambda,\wtd \lambda$ respectively, we have
			\begin{equation*}
				\gamma 
				\le \frac{\|X^{-1}\|_2\|\wtd X\|_2(\|A^{\pinv}\delta A\|_{\UI} + \|B^{-1}\delta B\|_{\UI})}{|x^HB\wtd x| - \|X^{-1}\|_2\|\wtd X\|_2\|B^{-1}\delta B\|_{\UI}}.
			\end{equation*}
	\end{enumerate}
	\label{thm:eigenvalue-relative-bound-(A,B):Bnonsingular}
\end{theorem}
\begin{proof}
	Using \eqref{eq:blockdetail}, define $R$ by
	$R = \bigoplus_{j=1}^{m}R_{j}$,
	where $R_{j} = \Omega_{j}\Lambda_{j}$,
	which is similar to that in the proof of Theorem~\ref{thm:subspace-relative-bound-(A,B):regular}.
	By \eqref{eq:eigen-structure}, similarly to the discussion in the proof of Theorem~\ref{thm:main_thm}, we have
	\begin{align*}
	R^HX^HB\wtd X -X^HB\wtd X\wtd R
	&=  - X^H\delta A\wtd X +X^H\delta B\wtd X\wtd R
	\\&=  - R^HX^HBA^{\pinv}\delta A\wtd X +X^H\delta B\wtd X\wtd R.
	\end{align*}
	Then, considering each block, for $j=1,\dots,m$ and $\wtd j=1,\dots,\wtd m$,
	\[
	R_{j}^HX_{j}^HB\wtd X_{\wtd j} - X_{j}^HB\wtd X_{\wtd j}\wtd R_{\wtd j}
	=  - R_{j}^HX_{j}^HBA^{\pinv}\delta A\wtd X_{\wtd j} + X_{j}^H\delta B\wtd X_{\wtd j}\wtd R_{\wtd j}
	.
	\]
	By $BX=X^{-H}\cdot(\bigoplus_{j=1}^{m}F_{n_{j}})$,
	\begin{equation}
		\label{eq:prf:tmp1}
	R_{j}^HX_{j}^HB\wtd X_{\wtd j} - X_{j}^HB\wtd X_{\wtd j}\wtd R_{\wtd j}
	=  - R_{j}^H		\what F_{n_{j}} X^{-1}A^{\pinv}\delta A\wtd X_{\wtd j} + \what F_{n_{j}}X^{-1}B^{-1}\delta B\wtd X_{\wtd j}\wtd R_{\wtd j}
	,
	\end{equation}
	where $\what F_{n_{j}}=X^HBX_{j}=\left[\begin{smallmatrix}
		\vdots\\
		F_{n_{j}}\\
		\vdots\\
\end{smallmatrix}\right]$, the matrix consisting of $F_{n_{j}}$ and possible zero blocks upper and lower.

Since $X_j^HB\wtd X_{\wtd j}$ is the solution to the structured Sylvester equation \eqref{eq:prf:tmp1}, which is of the form \eqref{eq:lm:structured-Sylvester:cont}, by Lemma~\ref{lm:structured-Sylvester}, we get
	\[
		\|X_j^HB\wtd X_{\wtd j}\|_F\le \phi_1\|\what F_{n_j}X^{-1}A^{\pinv}\delta A \wtd X_{\wtd j}\|_F + \phi_2\|\what F_{n_j}X^{-1}B^{-1}\delta B \wtd X_{\wtd j}\|_F,
	\] 
	where, noticing that
        $\gamma\le\frac{5-\sqrt{17}}{4}$,  $\wtd\gamma\le\frac{5-\sqrt{17}}{4}$ yield
        $\frac{1}{1-\gamma}\frac{1+\wtd\gamma}{1-\wtd\gamma}\le2$,
	\[
		\phi_1 
		= \tbinom{\wtd n_{\wtd j}+ n_j-1}{ n_j}\gamma^{1- n_j}\frac{1-\gamma^{ n_j}}{1-\gamma}\wtd\gamma^{-\wtd n_{\wtd j}}\frac{1+\wtd\gamma-2\wtd\gamma^{\wtd n_{\wtd j}}}{1-\wtd\gamma}
		\le 2\tbinom{\wtd n_{\wtd j}+ n_j-1}{ n_j}\gamma^{1- n_j}\wtd\gamma^{-\wtd n_{\wtd j}},
	\]
	and similarly $ \phi_2\le
        2\tbinom{\wtd n_{\wtd j}+ n_j-1}{\wtd n_{\wtd j}}\wtd\gamma^{1-\wtd n_{\wtd j}}\gamma^{- n_j}$.
	Thus,
	\[
		\|X_j^HB\wtd X_{\wtd j}\|_F\le 2\wtd\gamma^{-\wtd n_{\wtd j}}\gamma^{- n_j}\|X^{-1}\|_2\|\wtd X\|_2 [\gamma\tbinom{\wtd n_{\wtd j}+ n_j-1}{ n_j}\|A^{\pinv}\delta A\|_F + \wtd\gamma\tbinom{\wtd n_{\wtd j}+ n_j-1}{\wtd n_{\wtd j}}\|B^{-1}\delta B\|_F],
	\] 
	implying item~1.

	If $ n_j=\wtd n_{\wtd j}=1$, then
	left-multiplying by $e_{\sum_{i=1}^{j}n_{i}}^H$ and right-multiplying by
	$e_{\sum_{\wtd i=1}^{\wtd j}\wtd n_{\wtd i}}$ (or $e_{\sum_{\wtd i=1}^{\wtd j}\wtd n_{\wtd i}-1}$ to ensure that $\lambda,\wtd\lambda$ both lie in the upper half plane),
		and omitting subscripts of $e_*$,
	we find that the last equation yields
	\[
	\left(\lambda-\wtd\lambda\right)x^HB\wtd x
	=-\lambda e^HX^{-1}A^{\pinv}\delta A\wtd x
	.
	\]
	Then it follows that
\begin{equation}\label{eq:rel_eig1}
|\lambda - \wtd\lambda| \le \frac{|\lambda||e^HX^{-1}A^{\pinv}\delta A\wtd x| + |\wtd\lambda||e^HX^{-1}B^{-1}\delta B\wtd x|}{|x^HB\wtd x|}.
\end{equation}
By easy calculation, assuming that $\frac{|e^HX^{-1}B^{-1}\delta B\wtd x|}{|x^HB\wtd x|}<1$, this bound can be written as

\begin{equation}\label{eq:rel_eig2}
\frac{|\lambda - \wtd\lambda|}{|\lambda|} \le \displaystyle \frac{ \displaystyle \frac{|e^HX^{-1}A^{\pinv}\delta A\wtd x| + |e^HX^{-1}B^{-1}\delta B\wtd x|}{|x^HB\wtd x|}}{\displaystyle{1-\frac{|e^HX^{-1}B^{-1}\delta B\wtd x|}{|x^HB\wtd x|}}}.
\end{equation}

Since
\begin{align*}
	\frac{|e^HX^{-1}A^{\pinv}\delta A\wtd x|}{|x^HB\wtd x|}&\le \frac{\|X^{-1}\|_2\|\wtd X\|_2\|A^{\pinv}\delta A\|_{ui}}{|x^HB\wtd x|},
	\\
	\frac{|e^HX^{-1}B^{-1}\delta B\wtd x|}{|x^HB\wtd x|}&\le \frac{\|X^{-1}\|_2\|\wtd X\|_2\|B^{-1}\delta B\|_{ui}}{|x^HB\wtd x|},
\end{align*}
item~2 holds.
\end{proof}
\begin{remark}\label{rk:thm:eigenvalue-relative-bound-(A,B):Bnonsingular}
	As we can see, the bounds in Theorem~\ref{thm:eigenvalue-relative-bound-(A,B):Bnonsingular} are only relied on the corresponding blocks.
	Thus if different blocks share the same eigenvalue, then we may have a class of different bounds, in which we can choose the sharpest one.
\end{remark}

\section{Relative perturbation bound for a regular Hermitian QEP}\label{Sec4}
In this section, we will derive perturbation bounds for eigenvectors and eigenvalues for the QEP, using results in the previous section.

Given $M, C, K \in \mathbb{C}^{n\times n}$,
the corresponding quadratic matrix polynomial of order $n$ is defined by
\begin{equation}\label{def:QEP}
Q(\lambda)= \lambda^2 M + \lambda C+ K.
\end{equation}
$Q(\lambda)$ is called regular if $\det Q(\lambda)$ is not identically zero for $\lambda\in\mathbb{C}$ and singular otherwise. In this section we assume that $Q(\lambda)$ is regular. The QEP for $Q(\cdot)$ is to find $\lambda \in \mathbb{C}$ and nonzero $x\in \mathbb{C}$ such that
\begin{equation}\label{def:QEP1}
Q(\lambda)x=(\lambda^2 M  + \lambda C + K)x=0.
\end{equation}
When this equation is satisfied, $\lambda$ and $x$ are called an eigenvalue and an eigenvector, respectively.
We consider the QEP \eqref{def:QEP1}, where $M$ and $K$ are Hermitian nonsingular, and $C$ is Hermitian. Instead of the QEP \eqref{def:QEP1}, similarly as in \cite{TruMio}, we will consider the equivalent generalized eigenvalue problem
\[
\scrL_Q(\lambda)y=0, \quad y=\begin{bmatrix}x \\ \lambda x\end{bmatrix}\in \mathbb{C}^{2n},
\]
where $\scrL_Q(\lambda):=A-\lambda B$ is a matrix pair. For the case when $M$ and $K$ are nonsingular, we can obtain a symmetric linearization. For more details about linearizations of  QEPs, see \cite{TissMeer2001}. In this section we will consider Hermitian matrix pairs $(A,B)$, where
\begin{equation}\label{def:pair1}
A=\begin{bmatrix} - K & 0\\ 0 & M\end{bmatrix} \textrm{ and } B=\begin{bmatrix}C & M\\ M & 0\end{bmatrix}.
\end{equation}

The corresponding perturbed QEP is
\begin{equation}\label{def:QEP2}
\wtd \lambda^2 \wtd M \wtd x + \wtd \lambda \wtd C \wtd x+ \wtd K \wtd x=0,
\end{equation}
where $\wtd M=M+\delta M$ and $\wtd K=K+\delta K$ are Hermitian nonsingular  and $\wtd C=C+\delta C$ is Hermitian.
The corresponding perturbed pair $(\wtd A, \wtd B)$ is such that
 \begin{align}\label{def:pair2}
\wtd A=\begin{bmatrix} - (K+\delta K) & 0\\ 0 & M+\delta M\end{bmatrix} \textrm{ and } \wtd B=\begin{bmatrix}C+\delta C & M+\delta M\\ M+\delta M & 0\end{bmatrix}.
\end{align}
We assume that the matrix pair $(A,B)$ and also the perturbed pair $(\wtd A, \wtd B)$ satisfy assumptions (A1)--(A4). Then the result in the next theorem follows directly from Theorem \ref{thm:main_thm}.



\begin{theorem}
    Given $M$, $C$, $K\in \mathbb{C}^{n\times n}$ as in \eqref{def:QEP1}. Let $\wtd M=M+\delta M$, $\wtd C=C+\delta C$ and $\wtd K=K+\delta K$ be the corresponding perturbed matrices and $X$ and $\wtd X$ be the nonsingular matrices representing the associated eigenspaces of the exact and perturbed QEPs. Under the assumptions (A1)--(A4) on the linearized matrix pairs \eqref{def:pair1} and \eqref{def:pair2}, we have
	\begin{equation}
		\|\sin\Theta(\subspan(X_1),\subspan(\wtd X_1))\|_F
		\le \kappa_2(X)\kappa_2(\wtd X)\left(\phi_1^{\max} \delta a_F+\phi_2^{\max} \delta b_F\right),
		\label{eq:thm:subspace-relative-bound-(A,B):regular_HQEP:Bnonsingular:linear}
	\end{equation}
	\label{thm:subspace-relative-bound-(A,B):regular_HQEP}
	where $\phi_1^{\max},\phi_2^{\max}$ are the same as in \eqref{def:phi_12^m}, and
\begin{align*}
\delta a_F&=\sqrt{\|K^{-1}\delta K\|_F^2 + \|M^{-1}\delta M\|_F^2}\,,\\
\delta b_F&= \sqrt{2\cdot\|M^{-1}\delta M\|_F^2 + \|M^{-1}\delta C-M^{-1}CM^{-1}\delta M\|_F^2}.
\end{align*}
\end{theorem}
\begin{proof}
Bound \eqref{eq:thm:subspace-relative-bound-(A,B):regular_HQEP:Bnonsingular:linear} follows from Theorem \ref{thm:main_thm} simply by taking the Frobenius norm of the matrices
\begin{align*}
A^{-1}\delta A&=\begin{bmatrix}-K^{-1}\delta K & 0 \\ 0 & M^{-1}\delta M\end{bmatrix},\\
B^{-1}\delta B&=\begin{bmatrix}-M^{-1}\delta M & 0 \\ M^{-1}\delta C - M^{-1}CM^{-1}\delta M & M^{-1}\delta M\end{bmatrix}.
\qedhere
\end{align*}
\end{proof}

\begin{remark}
As a special case, we will consider the case when both eigenvalues $\lambda$ and $\wtd \lambda$ of matrix pairs \eqref{def:pair1} and \eqref{def:pair2} are semi-simple. Then the columns of the matrices  $X$ and $\wtd X$ are the vectors
 \begin{align}\label{def:vectX}
 X(:,i)=\begin{bmatrix} x_i\\ \lambda_i x_i\end{bmatrix} \quad\textrm{ and }\quad \wtd X(:,i)=\begin{bmatrix} \wtd x_i\\ \wtd \lambda_i \wtd x_i\end{bmatrix}, \quad i=1, \dots, 2n,
 \end{align}
 respectively,
 where $\lambda_i$, $\wtd \lambda_i$ are eigenvalues and $x_i$, $\wtd x_i \in \mathbb{C}^{n\times n}$ are the corresponding eigenvectors of the QEPs \eqref{def:QEP1} and \eqref{def:QEP2}.  In this case we can derive a bound for
$
 |\sin \vartheta (x_i, \wtd x_i)|,
$
 where $\vartheta$ is the angle between the eigenvectors $x_i$ and $\wtd x_i$. \\
 Let us assume that $x_i$ and $\wtd x_i$ are normalized, i.e., $\|x_i\|_2=\|\wtd x_i\|_2=1$. Then by the definition of the acute angle between two unit vectors, we have that $\cos \vartheta (x_i,\wtd x_i)=|x_i^H\wtd x_i|$ and also
\begin{align*}
\cos \vartheta(X(:,i),\wtd X(:,i)) &=\displaystyle\frac{|X(:,i)^H\wtd X(:,i)|}{\|X(:,i)\|_2\|\wtd X(:,i)\|_2}\\ \nonumber \qquad &=\frac{|1+\ol{\lambda_i}\wtd \lambda_i|}{\sqrt{1+|\lambda_i|^2}\sqrt{1+|\wtd\lambda_i|^2}}\cdot \cos \vartheta (x_i, \wtd x_i),
\end{align*}
where $\vartheta \in \left[0,\frac{\pi}{2}\right]$ and where $X(:,i)$ and $\wtd X(:,i)$ are defined in \eqref{def:vectX}. Since $|(1+\ol{\lambda_i}\wtd \lambda_i)|\le \sqrt{1+|\lambda_i|^2}\sqrt{1+|\wtd\lambda_i|^2}$, it is easy to see that
\begin{align*}
\cos \vartheta(X(:,i),\wtd X(:,i)) \le \cos \vartheta (x_i, \wtd x_i), \qquad \vartheta \in \left[0,\frac{\pi}{2}\right].
\end{align*}
This means that
\begin{align}\label{eq:ineqsin}
|\sin \vartheta (x_i, \wtd x_i)|\le |\sin \vartheta(X(:,i),\wtd X(:,i))| \textrm{ for } \vartheta \in \left[0,\frac{\pi}{2}\right],
\end{align}
and we  have the same bound \eqref{eq:thm:subspace-relative-bound-(A,B):regular_HQEP:Bnonsingular:linear} for the eigenvectors $x_i$ and $\wtd x_i$.
\end{remark}
The next theorem contains upper bound for the relative errors in the eigenvalues.
\begin{theorem}
Given $M$, $C$, $K\in \mathbb{C}^{n\times n}$ as in \eqref{def:QEP1}, let $\wtd M=M+\delta M$, $\wtd C=C+\delta C$ and $\wtd K=K+\delta K$ be corresponding perturbed matrices  and  $(A,B)$ and $(\wtd A, \wtd B)$ be the associated linearized Hermitian matrix pairs, respectively.
Suppose that $A,B,\wtd A,\wtd B,X,\wtd X$ satisfy the assumptions in Theorem~\ref{thm:eigenvalue-relative-bound-(A,B):Bnonsingular}.
Let $\lambda$ and $\wtd \lambda$ be the eigenvalues of the two matrix pairs, respectively, and $X_{\lambda}$ and $\wtd{X}_{\wtd \lambda}$ be corresponding subspaces spanned by $n_j$, $\wtd{n}_{\wtd j}$ vectors, respectively.
Suppose that $\lambda,\wtd \lambda$ are nonzero and both lie in the upper half plane, 
and write $\gamma=\left|\frac{\wtd \lambda - \lambda}{\lambda}\right|$, $\wtd \gamma=\left|\frac{\wtd \lambda - \lambda}{\wtd \lambda}\right|$.
Then:
\begin{enumerate}
\item provided $n_j=\wtd{n}_{\wtd j}=1$,  and rewriting the eigenvectors as $X_{\lambda}=\begin{bmatrix}x\\ \lambda x\end{bmatrix}$, $\wtd X_{\wtd \lambda}=\begin{bmatrix}\wtd x\\ \wtd \lambda \wtd x\end{bmatrix}$, where $x$ and $\wtd x$ are eigenvectors of the QEP \eqref{def:pair1} and \eqref{def:pair2}, respectively, it holds
	\begin{equation}\label{thm:eigenvalue-relative-bound-(A,B):BnonsingularQEP}
\gamma
	\le \frac{\|X^{-1}\|_2\|\wtd X\|_2(\delta a + \delta b + \delta c)}{\delta d - \|X^{-1}\|_2\|\wtd X\|_2(\delta b + \delta c)}
	,
	\end{equation}
where
\begin{align*}
	\delta a&= \max\left\{\|K^{-1}\delta K\|_2, \|M^{-1}\delta M\|_2\right\},\\
	\delta b&=  \|M^{-1}\delta M\|_2,\\
	\delta c&= \|M^{-1}\delta C\|_2 + \|M^{-1}CM^{-1}\delta M\|_2,\\
\delta d&= \left|x^H C \wtd x + (\wtd \lambda  + \bar{\lambda}) x^H M\wtd x\right|. 
\end{align*}
\item Provided $\gamma\le\frac{5-\sqrt{17}}{4},\wtd\gamma\le\frac{5-\sqrt{17}}{4}$, it holds
\begin{equation}\label{thm:eigenvalue-relative-bound-(A,B):BnonsingularQEP:Jordan}
\frac{\gamma^{n_j} {\wtd\gamma}^{\wtd{n}_{\wtd j}}}{\sqrt{\gamma^2+\wtd \gamma^2}}
	\le 2 \frac{\|X^{-1}\|_2\|{\wtd X}\|_2}{\|X_{\lambda}^H B {\wtd X}_{\wtd \lambda}\|_F}\sqrt{\tbinom{\wtd{n}_{\wtd j}+n_j-1}{n_j}^2 \delta a_F^2 + \tbinom{\wtd{n}_{\wtd j}+n_j-1}{\wtd{n}_{\wtd j}}^2\delta b_F^2}
	,
\end{equation}
where $\delta a_F$ and $\delta b_F$ are defined in
	Theorem~\ref{thm:subspace-relative-bound-(A,B):regular_HQEP}.
\end{enumerate}
\end{theorem}
\begin{proof}
Bounds \eqref{thm:eigenvalue-relative-bound-(A,B):BnonsingularQEP} and  \eqref{thm:eigenvalue-relative-bound-(A,B):BnonsingularQEP:Jordan} hold simply by Theorem \ref{thm:eigenvalue-relative-bound-(A,B):Bnonsingular}, using the facts that in item~1:
$$\|A^{-1}\delta A\|_2\le \delta a,\qquad \|B^{-1}\delta B\|_2 \le \delta b+\delta c,$$
\[
	\left|\begin{bmatrix}x^H & \ol{\lambda} x^H\end{bmatrix} B\begin{bmatrix}\wtd x\\ \wtd {\lambda} \wtd x\end{bmatrix}\right|=\left|x^H C \wtd x + \wtd \lambda x^H M \wtd x + \bar{\lambda} x^H M\wtd x\right|,
\]
while in item~2:
\[
	\|A^{-1}\delta A\|_F=\delta a_F \quad \textrm{and} \quad \|B^{-1}\delta B\|_F=\delta b_F.
	\qedhere
\]
\end{proof}
\section{Numerical examples}\label{Sec5}
In this section, the perturbation bounds for regular Hermitian QEPs given in Section~\ref{Sec4} will be illustrated by several numerical examples.
Compared to the bound for hyperbolic QEPs given in \cite{TruMio}, which can be considered as a special case here,
the new bound is not restricted to overdamped systems,
and can be applied to different mechanical systems described by nonsingular Hermitian mass $M$, nonsingular Hermitian stiffness $K$, and Hermitian damping $C$.

The numerical examples below can be grouped into two categories.
If possible, we will compare the new bound with the bound given in \cite[Theorem~7]{TruMio}.
Instead of the matrix pair $A-\lambda B$ given in \eqref{def:pair1}, the old bound is based on  the equivalent matrix pair $A_0- \dfrac{1}{\lambda} J$ where
\begin{align*}
A_0 &  = \begin{bmatrix} L_K^{-1} C L_K^{-H} &  L_K^{-1} L_M\\ L_M^H L_K^{-H} & 0 \end{bmatrix} , \quad
J  = \begin{bmatrix} -I & 0 \\0 & I \end{bmatrix},
\end{align*}
and the corresponding perturbations
\begin{align*}
\delta A &  = \begin{bmatrix} L_K^{-1} \delta C L_K^{-H} &   L_K^{-1} \delta M L_M^{-H} \\
 L_M^{-1} \delta M L_K^{-H} & 0 \end{bmatrix} , \quad
\delta J  = \begin{bmatrix} L_K^{-1} \delta K  L_K^{-H} & 0 \\0 & L_M^{-1} \delta M  L_M^{-H}\end{bmatrix} .
\end{align*}
Here, $L_M$ and $L_K$ are the Cholesky factors of $M$ and $K$, respectively, or equivalently $M=L_ML_M^H$ and $K=L_K L_K^H$.
Then the old bound is
given by
\begin{equation} \label{Hypersintheta}
| \sin{\vartheta(X(:,i), \wtd X(:,i))}|  \leq \kappa_2(X) \kappa_2(\wtd X)
 \left(\frac{\|{A^{-1}_0}\delta A_0\|_F}{ \min_{{j\neq i}}
\frac{|\lambda_i - \wtd \lambda_j|}{|\wtd \lambda_j|}}  +  \frac{\|J \delta J\|_F}{ \min_{{j\neq i}}
\frac{|\lambda_i - \wtd \lambda_j|}{|\lambda_i|}}\right),
\end{equation}
where $ i=1,\ldots,2n$.

Otherwise, if the bound given in \cite{TruMio} is not valid, we will illustrate that the new bound \eqref{eq:thm:subspace-relative-bound-(A,B):regular_HQEP:Bnonsingular:linear}  is applicable for the many cases when the bound  \eqref{Hypersintheta} is not.

\begin{example}\label{eg:wiresaw1}
This is the problem \verb|Wiresaw1| in the collection NLEVP \cite{nlevp}. It is a gyroscopic QEP arising in the vibration analysis of a wiresaw. More details can be found in \cite{Kao2000}.
It takes the form $G(\lambda)x=(\lambda^2 M + \lambda C + K)x=0$, where the coefficient matrices are defined by
\begin{gather*}
	M=\frac{1}{2}I_n, \quad K=\frac{\pi^2(1-\nu^2)}{2}\diag(j^2)_{j=1,\dots,n},\\
	C=-C^T=[c_{jk}]_{j,k=1,\dots,n},
	\quad \text{with}\quad c_{jk}=
	\begin{cases}
		\dfrac{4jk}{j^2-k^2}\nu, & \text{if $j+k$ is odd},\\
		0, & \text{otherwise.}
	\end{cases}
\end{gather*}
Here, $n$ is the size of the problem and $\nu$ is a real nonnegative parameter corresponding to the speed of the wire.
Clearly $M\succ0$, $K$ is definite Hermitian when $\nu\ne1$, and $C$ is skew-Hermitian.
Then the quadratic matrix polynomial
\begin{equation}\label{eq:gyro1}
	Q(\lambda):=-G(-\ii\lambda)= \lambda^2 M + \lambda (\ii C)- K
\end{equation}
is regular and Hermitian, which implies that our bound \eqref{eq:thm:subspace-relative-bound-(A,B):regular_HQEP:Bnonsingular:linear} can be applied.

Note that for $0 < \nu <1$, $K\succ0$ and \eqref{eq:gyro1} is hyperbolic (but not overdamped).
Then the corresponding linearization $A-\lambda B$ is positive definite and bound \eqref{Hypersintheta} for hyperbolic QEPs can also be applied.
On the contrary, for $\nu>1$, \eqref{eq:gyro1} is not hyperbolic, and then bound \eqref{Hypersintheta} is not valid.

First we choose $n=5,\nu=0.9<1$, and a group of random perturbations $\delta M$, $\delta C$ and $\delta K$ satisfying
\begin{equation}\label{perturb}
	|(\delta M)_{ij}| \le \eta |M_{ij}|, \quad |(\delta C)_{ij}| \le \eta |C_{ij}|, \quad |(\delta K)_{ij}| \le \eta |K_{ij}|,
\end{equation}
where $\eta=10^{-8}$ and matrices $M+\delta M$, $\ii(C + \delta C)$, $K+\delta K$ are also Hermitian.
Then we will compare our bound \eqref{eq:thm:subspace-relative-bound-(A,B):regular_HQEP:Bnonsingular:linear} and bound \eqref{Hypersintheta}
for the eigenvectors $x_1,\wtd x_1$ which correspond to the eigenvalue $\lambda_1= 21.9063$ of  \eqref{eq:gyro1}.
The new bound \eqref{eq:thm:subspace-relative-bound-(A,B):regular_HQEP:Bnonsingular:linear} yields
\[
|\sin \vartheta(x_1,\wtd x_1)| \le 4.9283 \cdot 10^{-5},
\]
and it is not much worse than the bound \eqref{Hypersintheta} which gives
\[
|\sin \vartheta(x_1,\wtd x_1)| \le 1.1908 \cdot 10^{-6}.
\]


Then we choose $n$ and $\delta M,\delta C,\delta K$ in the same way, but $\nu=1.0019>1$. 
For the eigenvectors $x_1,\wtd x_1$ that correspond to the eigenvalue $\lambda_1= -22.7864$, bound \eqref{eq:thm:subspace-relative-bound-(A,B):regular_HQEP:Bnonsingular:linear} yields
\[
|\sin \vartheta(x_1,\wtd x_1)| \le 8.5756 \cdot 10^{-6},
\]
while bound \eqref{Hypersintheta} cannot be applied.
In comparison, the exact value is
\[
|\sin \vartheta(x_1,\wtd x_1)| \approx 1.7615\cdot 10^{-9}.
\]
\end{example}

\begin{example}\label{eg:brake}
This example is related to the analysis of the behavior of brake systems and is taken from \cite{model}.
Brake squeal is a major problem in the automotive industry and it is based on the loss of stability of the brake system.
In this example, we will consider negative-friction damping  excitation mechanisms.
Instability of the system is caused by a damping matrix $C\prec0$.

In \cite{model}, the authors consider mechanical models of dimension $1$ or $2$.
Here this problem is generalized to a QEP of size $n$
\begin{equation}\label{QEPexample}(\lambda^2 M  + \lambda C  + K)x=0,\end{equation}
where $M$, $C$ and $K$ are mass, damping and stiffness matrices, respectively, and defined by
\[
	M=\diag(j)_{j=1,\dots,n},\quad
	C=-\gamma I_n,\quad
	K=
	\begin{bmatrix}
		10 & -5     &&\\
		-5 & \ddots & \ddots &\\
		& \ddots & \ddots & -5\\
		&        & -5     & 10
	\end{bmatrix}_{n\times n}.
\]
Clearly, $M\succ0, K\succ0, C\prec0$.

Here, we choose $n=4,\gamma=0.1$ and random perturbations $\delta M$, $\delta C$ and $\delta K$ as in \eqref{perturb}.
Also, we assume that the perturbed matrices satisfy\\ $\wtd M=M+\delta M\succ0, \wtd K=K+\delta K\succ0, \wtd C = C+ \delta C\prec0$.

Contrary to the bounds in \cite{TruMio} for hyperbolic QEPs whose eigenvalues are real,
the new bounds can be applied to the complex eigenvalues and their corresponding eigenvectors.
This example is designed to illustrate the sensitivity of complex eigenvalues and corresponding eigenvectors for QEPs with negative damping.

\textbf{Part 1:} For the eigenvectors $x_1,\wtd x_1$ which correspond to the eigenvalue $\lambda_1 = 0.0251 + 1.1701\ii$,  bound
\eqref{eq:thm:subspace-relative-bound-(A,B):regular_HQEP:Bnonsingular:linear} 
gives
\[
|\sin \vartheta(x_1,\wtd x_1)| \le    6.9547\cdot 10^{-6},
\]
in comparison with the exact value
\[
	|\sin \vartheta(x_1,\wtd x_1)| \approx    1.2738 \cdot 10^{-8}.
\]

\textbf{Part 2:}
In this part, we will illustrate the performance of our eigenvalue bound \eqref{thm:eigenvalue-relative-bound-(A,B):BnonsingularQEP}.
In comparison to the bounds for eigenvalues given in \cite{Veselic2011} and \cite{TruMio},
which hold only for overdamped QEPs, our new bound is applicable to any regular Hermitian QEP. Table \ref{tabla} shows our bound and the exact relative error for all eigenvalues 
$\lambda_i \in \Lambda(Q)$, which appear in complex conjugate pairs.
\begin{table}[ht]
	\centering
\begin{tabular}{c|c|c}
$\lambda_i$ &   exact value  & estimate \eqref{thm:eigenvalue-relative-bound-(A,B):BnonsingularQEP}\\ \hline
0.0438$\pm$3.4550$\ii$ & 4.8239e-09 & 6.9836e-07 \\
0.0224$\pm$2.3223$\ii$ & 6.0146e-10 & 3.3994e-07 \\
0.0197$\pm$1.6640$\ii$ & 7.8488e-10 & 2.5006e-07 \\
0.0183$\pm$0.8543$\ii$ & 3.0628e-09 & 3.3075e-07
\end{tabular}
\caption{Example~\ref{eg:brake}: relative perturbation bound \eqref{thm:eigenvalue-relative-bound-(A,B):BnonsingularQEP} and exact relative error for eigenvalues.} \label{tabla}
\end{table}
\end{example}

\begin{example}\label{eg:Jordan-block}
This example is constructed to show that our bound for eigenvalues and eigenvectors is also sensitive to perturbations of Jordan blocks if they appear in the canonical form of the matrix pair $(A,B)$.
In this small academic example, $M$, $C$ and $K$ are chosen as:
\begin{equation*}
M=\left[\begin{array}{cc} 1 & 0\\ 0& 2\end{array}\right], \quad K=\left[\begin{array}{cc} 2 & 0\\ 0& 2\end{array}\right] \quad \textrm{and} \quad C=2\cdot K.
\end{equation*}
The Jordan form of the linearized pair $(A,B)$ is
\begin{equation*}
J=\diag\left\{-3.4142, -0.5858, \begin{bmatrix}-1 & 1 \\ 0 &-1\end{bmatrix}\right\}.
\end{equation*}
Under small perturbations $\delta M$, $\delta C$ and $\delta K$ as in \eqref{perturb}, where $\mu=10^{-7}$, we lose the size-two Jordan block of the matrix pair $(A,B)$ and all the eigenvalues of the perturbed matrix pair $(\wtd A, \wtd B)$ are semi-simple.
Table \ref{tabl1} shows that our bound in comparison to the exact value of the relative error in eigenvalues is good enough to detect the case that the structure of Jordan blocks is changed. The bounds given in \cite{Veselic2011}  and \cite{TruMio} cannot be applied here.
\begin{table}[ht]
	\centering
\begin{tabular}{c||c|c}
$\lambda_i$ &   exact value  & estimate \eqref{thm:eigenvalue-relative-bound-(A,B):BnonsingularQEP}\\ \hline\hline
-3.4142 &  2.2188e-07 &  2.7405e-05 \\
 -0.5858&  3.8069e-08 & 8.0671e-07 \\ \hline\hline
$\lambda_i$ & exact value & estimate \eqref{thm:eigenvalue-relative-bound-(A,B):BnonsingularQEP:Jordan} \\ \hline
-1  &  5.6518e-08 & 5.2899e-06
\end{tabular}
\caption{Example~\ref{eg:Jordan-block}: relative perturbation bounds \eqref{thm:eigenvalue-relative-bound-(A,B):BnonsingularQEP} and \eqref{thm:eigenvalue-relative-bound-(A,B):BnonsingularQEP:Jordan} and exact relative errors for eigenvalues.} \label{tabl1}
\end{table}

%

The eigenvector perturbation bound \eqref{eq:thm:subspace-relative-bound-(A,B):regular_HQEP:Bnonsingular:linear} cannot directly for the eigenvectors of the QEP, 
so we will measure the distance between the subspaces.
The subspace that we will consider is spanned by the columns of the matrices $X_1=X(:,1\!\!:\!\!2)$ of which one is the eigenvector and the other is the generalized eigenvector for the eigenvalue $\lambda_{1,2}=-1$.
The perturbed subspace is spanned by the columns of the matrix $\wtd X_1=\wtd X(:,1\!\!:\!\!2)$ which are the corresponding eigenvectors for two distinct eigenvalues.
Our bound \eqref{eq:thm:subspace-relative-bound-(A,B):regular_HQEP:Bnonsingular:linear}  gives
\begin{align*}
\|\sin\Theta(\subspan(X_1), \subspan(\wtd X_1))\|_F \le   7.9620 \cdot 10^{-2},
\end{align*}
since $\kappa_2(X)\approx 3.4558$, $\kappa_2(\wtd X) \approx 1.2357\cdot 10^4$, $\phi_1^{\max} \approx 2.4142$ and $\phi_2^{\max} \approx 2.4143 $.
The exact distance is
\begin{align*}
\|\sin\Theta(\subspan(X_1), \subspan(\wtd X_1))\|_F \approx   8.0927 \cdot 10^{-5},
\end{align*}
which confirms the fact that Jordan blocks are very sensitive to small perturbations and small perturbations can significantly change the corresponding invariant subspace.


\end{example}

\section{Conclusions}\label{Sec6}
The main contributions of this paper are new relative perturbation bounds for
the eigenvalues (and their corresponding invariant subspaces) of regular
Hermitian quadratic eigenvalue problems, based on the new corresponding bounds
for regular Hermitian pairs.
The obtained bounds can be applied to many interesting problems, for example, to
quadratic eigenvalue problems appearing in many mechanical models,
especially models with indefinite damping or mass matrices.
The main advantage of the new bounds, compared to earlier bounds,
is that they are more general and can be applied not only to hyperbolic
quadratic eigenvalue problems, but also to other regular quadratic eigenvalue
problems. The quality of our bounds has been illustrated using several numerical
examples.

\appendix
\section{Proof of Lemma \ref{lm:structured-Sylvester}} \label{sec:proof-lm:structured-Sylvester}
In the very beginning, we point out our way to prove the lemma.
The key is to vectorize the equations \eqref{eq:lm:structured-Sylvester:cont} and \eqref{eq:lm:structured-Sylvester:disc}.
Then we use standard but complicated calculation to estimate the upper bound of the solution.

	Consider \eqref{eq:lm:structured-Sylvester:cont}.
	Recall that $(\Lambda,\Omega)$ and $(\Lambda',\Omega')$ are of type \textbf{R2}, \textbf{R3} or \textbf{R4}.
	Then $\Omega'\Lambda'$ and $\Omega\Lambda$ are lower triangular.
	Thus, $(I\otimes \Omega'\Lambda' - \Omega\Lambda\otimes I)$ is a lower triangular matrix.
	Each diagonal entry is $\lambda-\lambda'$, so this matrix is nonsingular. We have
	\begin{align*}
		\vectorize(Y)
		&=(I\otimes \Omega'\Lambda' - \Omega\Lambda\otimes I)^{-H}[
		-(I\otimes \Omega'\Lambda')^H\vectorize(M)
		+(\Omega\Lambda\otimes I)^H\vectorize(N)]
		\\&=: -W_1^H\vectorize(M) +  W_2^H\vectorize(N)
		,
	\end{align*}
	where $W_1=(I\otimes \Omega'\Lambda')(I\otimes \Omega'\Lambda' - \Omega\Lambda\otimes I)^{-1},
	W_2=(\Omega\Lambda\otimes I)(I\otimes \Omega'\Lambda' - \Omega\Lambda\otimes I)^{-1}$.
	This shows \eqref{eq:lm:structured-Sylvester:cont} has a unique solution $Y$, and $W_1-W_2=I$.
	Moreover,
	\[
	\|Y\|_F \le \|W_1\|_2\|M\|_F+\|W_2\|_2\|N\|_F.
	\]
	Compared to the result, the rest of the proof is to estimate $\|W_1\|_2$ and $\|W_2\|_2$.

	Considering the types of $(\Lambda,\Omega)$ and $(\Lambda',\Omega')$, there are eight cases to discuss.
	Before that, first we explicitly invert a structured matrix, which will be used several times in the following.
	Consider the block matrix
	$
	P = D_1\otimes L + FD_2G\otimes I  = [P_{i,j}]_{i,j=1,\dots,m}
	$
	with $D_1=\diag(d_{1,1},\dots,d_{1,m}),D_2=\diag(d_{2,1},\dots,d_{2,m})$ and $L\in\mathbb{C}^{n\times n}$ lower triangular.
	Specifically,
	\begin{align*}
		P_{i,j} &=
		\begin{cases}
			0, & i<j \text{\;or\;} i>j+1, \\
			d_{1,j}L, & i=j, \\
			d_{2,m-j}I, & i=j+1.
		\end{cases}
	\end{align*}
	Assume that $D_1$ and $L$ are nonsingular, and hence $P_{i,j}$ and also $P$ are nonsingular.
	Then its inverse $P^{-1}=[Q_{i,j}]_{i,j=1,\dots,m}$ satisfies
	\[
		Q_{i,j} =
		\begin{cases}
			\prod\limits_{k=1}^{i-j}(-P_{i-k+1,i-k+1}^{-1}P_{i-k+1,i-k})Q_{j,j}, & i> j, \\
			0, & i<j.
		\end{cases}
	\]
	Since $P$ is block lower triangular, $Q_{j,j}=d_{1,j}^{-1}L^{-1}$.
	Calculating recursively, we have\footnote{Actually the left-hand side is $Q_{i,j}$, but we abuse the notation $Q$ to denote the matrix.}
	\begin{equation}\label{eq:Q}
		Q =
		\begin{cases}
			(-1)^{i-j}\prod\limits_{k=1}^{i-j+1} d_{1,i-k+1}^{-1}\prod\limits_{k=1}^{i-j} d_{2,m-i+k}L^{j-i-1}, & i\ge j, \\
			0, & i<j.
		\end{cases}
	\end{equation}

	Now we can go back for the eight cases.
	\begin{enumerate}
		\item \label{itm:structured-Sylvester:cont:33}
			$(\Lambda,\Omega)$ and $(\Lambda',\Omega')$ are both of type \typ{3}:
			\begin{align*}
				I\otimes \Omega'\Lambda' - \Omega\Lambda\otimes I
				&= I\otimes (\lambda'I+|\lambda'|FG) - (\lambda I+|\lambda|FG)\otimes I
				\\&= I\otimes ([\lambda'-\lambda]I+|\lambda'|FG) - |\lambda|FG\otimes I,
			\end{align*}
			and then by \eqref{eq:Q}
			\begin{equation*}
				(I\otimes \Omega'\Lambda' - \Omega\Lambda\otimes I)^{-1}
				=
				\begin{cases}
					|\lambda|^{i-j}([\lambda'-\lambda]I+|\lambda'|FG)^{j-i-1},  & i\ge j, \\
					0, & i<j,
				\end{cases}
			\end{equation*}
			where also by \eqref{eq:Q}
			\begin{equation*}
				([\lambda'-\lambda]I+|\lambda'|FG)^{-1}=
				\begin{cases}
					(-|\lambda'|)^{i'-j'}(\lambda'-\lambda)^{j'-i'-1} , & i'\ge j', \\
					0, & i'<j'.
				\end{cases}
			\end{equation*}
			Then
			\begin{equation}
				([\lambda'-\lambda]I+|\lambda'|FG)^{j-i-1}=
				\begin{cases}
					 \binom{i'-j'+i-j+1}{i-j+1}(-|\lambda'|)^{i'-j'}(\lambda'-\lambda)^{j'-i'+j-i-1} , & i'\ge j', \\
					0, & i'<j'.
				\end{cases}
			\end{equation}
			Thus
			\begin{equation*}
				W_1 =
				\begin{cases}
					 |\lambda|^{i-j}(\lambda'I+|\lambda'|FG)([\lambda'-\lambda]I+|\lambda'|FG)^{j-i-1},  & i\ge j, \\
					0, & i<j,
				\end{cases}
			\end{equation*}
			where
			\begin{multline*}
				(\lambda'I+|\lambda'|FG)([\lambda'-\lambda]I+|\lambda'|FG)^{j-i-1}
				\\=
				\begin{cases}
					 (-|\lambda'|)^{i'-j'}(\lambda'-\lambda)^{j'-i'+j-i}[\binom{i'-j'+i-j+1}{i-j+1}\lambda'(\lambda'-\lambda)^{-1}- \binom{i'-j'+i-j}{i-j+1}], & i'\ge j', \\
					0, & i'<j'.
				\end{cases}
			\end{multline*}
			It is easy to see $ \|W_1\|_1 = \|W_1\|_\infty=\|W_1e_1\|_1$.
			Note that $\|W_1\|_2\le\sqrt{\|W_1\|_1\|W_1\|_\infty}=\|W_1\|_1$.
			Thus,
			\begin{align*}
				\|W_1\|_2
				&\le\|W_1e_1\|_1
				\\&= \sum_{i=1}^{n}\sum_{i'=1}^{n'}\left| |\lambda|^{i-1} (-|\lambda'|)^{i'-1}(\lambda'-\lambda)^{2-i'-i}[\tbinom{i'+i-1}{i}\lambda'(\lambda'-\lambda)^{-1}- \tbinom{i'+i-2}{i}]\right|
				\\&\le \sum_{i=1}^{n}\sum_{i'=1}^{n'}\gamma^{1-i}\gamma'^{1-i'}[\tbinom{i'+i-1}{i}\gamma'^{-1}+ \tbinom{i'+i-2}{i}]
				\\&= 2\sum_{i=1}^{n}\sum_{i'=1}^{n'-1}\gamma^{1-i}\gamma'^{-i'}\tbinom{i'+i-1}{i} + \sum_{i=1}^{n}\gamma^{1-i}\gamma'^{-n'}\tbinom{n'+i-1}{i}
				\\&\le \left(2\sum_{i=1}^{n}\sum_{i'=1}^{n'-1}\gamma^{1-i}\gamma'^{-i'} + \sum_{i=1}^{n}\gamma^{1-i}\gamma'^{-n'}\right)\tbinom{n'+n-1}{n}
				\\&= \left(2\gamma'^{-1}\frac{1-\gamma'^{1-n'}}{1-\gamma'^{-1}}\frac{1-\gamma^{-n}}{1-\gamma^{-1}} + \gamma'^{-n'}\frac{1-\gamma^{-n}}{1-\gamma^{-1}}\right)\tbinom{n'+n-1}{n}
				\\&= \gamma'^{-1}\frac{2-\gamma'^{1-n'} - \gamma'^{-n'}}{1-\gamma'^{-1}}\frac{1-\gamma^{-n}}{1-\gamma^{-1}}\tbinom{n'+n-1}{n}
				\\&= \gamma'^{-n'}\gamma^{1-n}\frac{1+\gamma'-2\gamma'^{n'}}{1-\gamma'}\frac{1-\gamma^{n}}{1-\gamma}\tbinom{n'+n-1}{n}
				\\&= \tbinom{n'+n-1}{n}\varphi_-(\gamma,n)\varphi_+(\gamma',n').
			\end{align*}
			Similarly,
			\[
			\|W_2\|_2\le \tbinom{n'+n-1}{n'}\varphi_-(\gamma',n')\varphi_+(\gamma,n).
			\]
		\item \label{itm:structured-Sylvester:cont:44}
			$(\Lambda,\Omega)$ and $(\Lambda',\Omega')$ are both of type \typ{4}:
			\begin{equation*}
			\begin{aligned}
				I\otimes \Omega'\Lambda' - \Omega\Lambda\otimes I
				&= \begin{bmatrix} I&\\ &I\end{bmatrix}\otimes \begin{bmatrix} \lambda'I+|\lambda'|FG&\\ &\ol{\lambda'}I+|\ol{\lambda'}|FG\end{bmatrix} \\ &- \begin{bmatrix} \lambda I+|\lambda|FG&\\ &\ol{\lambda}I+|\ol{\lambda}|FG\end{bmatrix}\otimes \begin{bmatrix} I&\\ &I\end{bmatrix}
				\\&=  \diag(R_1,R_2,\ol{R_1},\ol{R_2}),
			\end{aligned}
			\end{equation*}
			where
			\begin{align*}
				R_1 &= [I\otimes (\lambda'I+|\lambda'|FG) - (\lambda I+|\lambda|FG)\otimes I], \\
				R_2 &= [I\otimes (\ol{\lambda'}I+|\lambda'|FG) - (\lambda I+|\lambda|FG)\otimes I].
			\end{align*}
			Then
			\begin{equation*}
			\begin{aligned}
			W_1 = \diag(R_1^{-1}(\lambda'I+|\lambda'|FG),& R_2^{-1}(\ol{\lambda'}I+|\lambda'|FG),\\ &\ol{R_1}^{-1}(\lambda'I+|\lambda'|FG),\ol{R_2}^{-1}(\ol{\lambda'}I+|\lambda'|FG)).
			\end{aligned}
			\end{equation*}
			Write $\gamma_1 := \left|\frac{\lambda'-\lambda}{\lambda}\right|,\gamma_1' := \left|\frac{\lambda'-\lambda}{\lambda'}\right|,
			\gamma_2 := \left|\frac{\ol{\lambda'}-\lambda}{\lambda}\right|,\gamma_2' := \left|\frac{\ol{\lambda'}-\lambda}{\lambda'}\right|$, then
			by the same calculation in case~\ref{itm:structured-Sylvester:cont:33} we have
			\[
			\|R_1^{-1}(\lambda'I+|\lambda'|FG)\|_2 \le \tbinom{n'+n-1}{n}\varphi_-(\gamma_1,n)\varphi_+(\gamma_1',n');
			\]
			\[
			\|R_2^{-1}(\ol{\lambda'}I+|\lambda'|FG)\|_2 \le \tbinom{n'+n-1}{n}\varphi_-(\gamma_2,n)\varphi_+(\gamma_2',n').
			\]
			Similarly to the calculation in case~\ref{itm:structured-Sylvester:cont:33} we get
			\begin{equation*}
				\ol{R_1}^{-1}(\lambda'I+|\lambda'|FG) =
				\begin{cases}
					 |\lambda|^{i-j}(\lambda'I+|\lambda'|FG)([\ol{\lambda'}-\ol{\lambda}]I+|\lambda'|FG)^{j-i-1},  & i\ge j, \\
					0, & i<j,
				\end{cases}
			\end{equation*}
			where
			\begin{multline*}
				 (\lambda'I+|\lambda'|FG)([\ol{\lambda'}-\ol{\lambda}]I+|\lambda'|FG)^{j-i-1}
				\\=
				\begin{cases}
					 (-|\lambda'|)^{i'-j'}(\ol{\lambda'}-\ol{\lambda})^{j'-i'+j-i}[\binom{i'-j'+i-j+1}{i-j+1}\lambda'(\ol{\lambda'}-\ol{\lambda})^{-1}- \binom{i'-j'+i-j}{i-j+1}], & i'\ge j', \\
					0, & i'<j',
				\end{cases}
			\end{multline*}
			and then
			\[
			\|\ol{R_1}^{-1}(\lambda'I+|\lambda'|FG)\|_2 \le \tbinom{n'+n-1}{n}\varphi_-(\gamma_1,n)\varphi_+(\gamma_1',n');
			\]
			similarly,
			\[
			\|\ol{R_2}^{-1}(\ol{\lambda'}I+|\lambda'|FG)\|_2 \le \tbinom{n'+n-1}{n}\varphi_-(\gamma_2,n)\varphi_+(\gamma_2',n').
			\]
			In summary,
			\[
			\|W_1\|_2\le \tbinom{n'+n-1}{n}\varphi_-(\gamma,n)\varphi_+(\gamma',n')\,,
			\]
			where $\gamma=\min\{\gamma_1,\gamma_2\}$.
			Similarly,
			\[
			\|W_2\|_2\le \tbinom{n'+n-1}{n'}\varphi_-(\gamma',n')\varphi_+(\gamma,n).
			\]
		\item \label{itm:structured-Sylvester:cont:23}
			$(\Lambda,\Omega)$ is of type \typ{2}, while $(\Lambda',\Omega')$ is of type \typ{3}:
			\begin{equation*}
				I\otimes \Omega'\Lambda' - \Omega\Lambda\otimes I
				= I\otimes (\lambda'I+|\lambda'|FG) - FG\otimes I,
			\end{equation*}
			and then
			\begin{equation*}
				(I\otimes \Omega'\Lambda' - \Omega\Lambda\otimes I)^{-1}
				=
				\begin{cases}
					(\lambda'I+|\lambda'|FG)^{j-i-1},  & i\ge j, \\
					0, & i<j,
				\end{cases}
			\end{equation*}
			where
			\begin{equation*}
				(\lambda'I+|\lambda'|FG)^{-1}=
				\begin{cases}
					(-|\lambda'|)^{i'-j'}\lambda'^{j'-i'-1} , & i'\ge j', \\
					0, & i'<j'.
				\end{cases}
			\end{equation*}
			Thus
			\begin{equation*}
				W_1 =
				\begin{cases}
					(\lambda'I+|\lambda'|FG)^{j-i},  & i\ge j, \\
					0, & i<j,
				\end{cases}
			\end{equation*}
			where
			\begin{equation*}
				(\lambda'I+|\lambda'|FG)^{j-i}
				=
				\begin{cases}
					(-|\lambda'|)^{i'-j'}\lambda'^{j'-i'+j-i}\binom{i'-j'+i-j}{i-j}, & i'\ge j', \\
					0, & i'<j',
				\end{cases}
				\quad \text{for $i>j$}.
			\end{equation*}
			It is easy to see $ \|W_1\|_1 = \|W_1\|_\infty=\|W_1e_1\|_1$.
			Note that $\|W_1\|_2\le\sqrt{\|W_1\|_1\|W_1\|_\infty}=\|W_1\|_1$.
			Thus,
			\begin{align*}
				\|W_1\|_2
				&\le\|W_1e_1\|_1
				\\&= 1+\sum_{i=2}^{n}\sum_{i'=1}^{n'}\left|(-|\lambda'|)^{i'-1}\lambda'^{2-i'-i}\tbinom{i'+i-2}{i-1}\right|
				\\&= 1+\sum_{i=2}^{n}\sum_{i'=1}^{n'}|\lambda'|^{1-i}\tbinom{i'+i-2}{i-1}.
			\end{align*}
			Noticing\footnote{$\tbinom{p}{q-1} + \tbinom{p}{q}=\tbinom{p+1}{q}$ for any $q\in\mathbb{N}_+$.} $
			\sum_{i'=1}^{n'}\tbinom{i'+i-2}{i-1}
			=\tbinom{i-1}{i-1} + \sum_{i'=2}^{n'}\tbinom{i'+i-2}{i-1}
			=\tbinom{i}{i} + \sum_{i'=2}^{n'}\tbinom{i'+i-2}{i-1}
			= \dots
			= \tbinom{n'+i-1}{i}
			$,
			\begin{align*}
				\|W_1\|_2
				&\le 1+\sum_{i=2}^{n}|\lambda'|^{1-i}\tbinom{n'+i-1}{i}
				\\&\le 1+\frac{|\lambda'|^{-1}-|\lambda'|^{-n}}{1-|\lambda'|^{-1}}\tbinom{n'+n-1}{n}
				\\&= 1+\tbinom{n'+n-1}{n}|\lambda'|^{-1}\varphi_-(|\lambda'|,n-1).
			\end{align*}
			Then
			\begin{align*}
				\|W_2\|_2
				&\le\|W_1e_1-e_1\|_1
				\\&= \sum_{i=2}^{n}\sum_{i'=1}^{n'}\left|(-|\lambda'|)^{i'-1}\lambda'^{2-i'-i}\tbinom{i'+i-2}{i-1}\right|
				\\&= \tbinom{n'+n-1}{n}|\lambda'|^{-1}\varphi_-(|\lambda'|,n-1).
			\end{align*}
		\item \label{itm:structured-Sylvester:cont:24}
			$(\Lambda,\Omega)$ is of type \typ{2}, while $(\Lambda',\Omega')$ is of type \typ{4}:
			\begin{align*}
				I\otimes \Omega'\Lambda' - \Omega\Lambda\otimes I
				&= I\otimes \begin{bmatrix} \lambda'I+|\lambda'|FG&\\ &\ol{\lambda'}I+|\ol{\lambda'}|FG\end{bmatrix} -  FG\otimes \begin{bmatrix} I&\\ &I\end{bmatrix}
				\\&=  \diag(R_1,\ol{R_1}),
			\end{align*}
			where
			\begin{equation*}
				R_1 = [I\otimes (\lambda'I+|\lambda'|FG) - FG\otimes I].
			\end{equation*}
			Then
			\[
			W_1 = \diag(R_1^{-1}(\lambda'I+|\lambda'|FG),\ol{R_1}^{-1}(\ol{\lambda'}I+|\lambda'|FG)).
			\]
			By the same calculation in case~\ref{itm:structured-Sylvester:cont:23},
			\[
			\|R_1^{-1}(\lambda'I+|\lambda'|FG)\|_2 \le \tbinom{n'+n-1}{n}\varphi_-(|\lambda'|,n);
			\]
			Note that
			\[
			\|\ol{R_1}^{-1}(\ol{\lambda'}I+|\lambda'|FG)\|_2 = \|R_1^{-1}(\lambda'I+|\lambda'|FG)\|_2.
			\]
			Thus
			\[
			\|W_1\|_2\le \tbinom{n'+n-1}{n}\varphi_-(|\lambda'|,n).
			\]
			Similarly,
			\[
			\|W_2\|_2\le \tbinom{n'+n-1}{n}\varphi_-(|\lambda'|,n')-1.
			\]
		\item \label{itm:structured-Sylvester:cont:34}
			$(\Lambda,\Omega)$ is of type \typ{3}, while $(\Lambda',\Omega')$ is of type \typ{4}:
			\begin{align*}
				I\otimes \Omega'\Lambda' - \Omega\Lambda\otimes I
				&= I\otimes \begin{bmatrix} \lambda'I+|\lambda'|FG&\\ &\ol{\lambda'}I+|\ol{\lambda'}|FG\end{bmatrix} -  (\lambda I+|\lambda|FG)\otimes \begin{bmatrix} I&\\ &I\end{bmatrix}
				\\&=  \diag(R_1,\ol{R_1}),
			\end{align*}
			where
			\begin{equation*}
				R_1 = [I\otimes (\lambda'I+|\lambda'|FG) - (\lambda I+|\lambda|FG)\otimes I].
			\end{equation*}
			Then
			\[
			W_1 = \diag(R_1^{-1}(\lambda'I+|\lambda'|FG),\ol{R_1}^{-1}(\ol{\lambda'}I+|\lambda'|FG)).
			\]
			By the same calculation in case~\ref{itm:structured-Sylvester:cont:33},
			\[
			\|R_1^{-1}(\lambda'I+|\lambda'|FG)\|_2 \le \tbinom{n'+n-1}{n}\varphi_-(\gamma,n)\varphi_+(\gamma',n');
			\]
			Note that
			\[
			\|\ol{R_1}^{-1}(\ol{\lambda'}I+|\lambda'|FG)\|_2 = \|R_1^{-1}(\lambda'I+|\lambda'|FG)\|_2.
			\]
			Thus
			\[
			\|W_1\|_2\le \tbinom{n'+n-1}{n}\varphi_-(\gamma,n)\varphi_+(\gamma',n').
			\]
			Similarly,
			\[
			\|W_2\|_2\le \tbinom{n'+n-1}{n'}\varphi_-(\gamma',n')\varphi_+(\gamma,n).
			\]
		\item \label{itm:structured-Sylvester:cont:32}
			$(\Lambda,\Omega)$ is of type \typ{3}, while $(\Lambda',\Omega')$ is of type \typ{2}:
			Note that there exist two permutation matrices $P,Q$ such that $A\otimes B = P(B\otimes A)Q$. Thus, $\|A\otimes B\|_2=\|B\otimes A\|_2$.
			So it is the same as case~\ref{itm:structured-Sylvester:cont:23}.
		\item \label{itm:structured-Sylvester:cont:42}
			$(\Lambda,\Omega)$ is of type \typ{4}, while $(\Lambda',\Omega')$ is of type \typ{2}:
			the same as case~\ref{itm:structured-Sylvester:cont:24}.
		\item \label{itm:structured-Sylvester:cont:43}
			$(\Lambda,\Omega)$ is of type \typ{4}, while $(\Lambda',\Omega')$ is of type \typ{3}:
			the same as case~\ref{itm:structured-Sylvester:cont:34}.
	\end{enumerate}

	Consider \eqref{eq:lm:structured-Sylvester:disc}.
	Recall that $(\Lambda,\Omega)$ is of type \textbf{R2}, \textbf{R3} or \textbf{R4}.
	Then $\Omega\Lambda$ are lower triangular as well as $FG$.
	Thus, $(I\otimes I - \Omega\Lambda\otimes FG)$ is a lower triangular matrix.
	Each diagonal entry is $1$,	so this matrix is nonsingular.
	Then
	\begin{align*}
		\vectorize(Y)
		&=(I\otimes I - \Omega\Lambda\otimes FG)^{-H}[
		-(I\otimes I)^H\vectorize(M)
		+(\Omega\Lambda\otimes FG)^H\vectorize(N)]
		\\&=: -W_1^H\vectorize(M) +  W_2^H\vectorize(N)
		,
	\end{align*}
	where $W_1=(I\otimes I)(I\otimes I - \Omega\Lambda\otimes FG)^{-1},
	W_2=(\Omega\Lambda\otimes FG)(I\otimes I - \Omega\Lambda\otimes FG)^{-1}$.
	Similarly to the discussion on \eqref{eq:lm:structured-Sylvester:cont},
	\eqref{eq:lm:structured-Sylvester:disc} has a unique solution $Y$, $W_1-W_2=I$, and the rest of the proof is to estimate $\|W_1\|_2$ and $\|W_2\|_2$.

	Note that there exist two permutation matrices $P,Q$ to make $A\otimes B = P(B\otimes A)Q$. Thus, $\|A\otimes B\|_2=\|B\otimes A\|_2$.
	Then
	\[
	I\otimes I - FG\otimes \Omega\Lambda =
	\begin{bmatrix}
		I             &        &               &   \\
		\Omega\Lambda & I      &               &   \\
		              & \ddots & \ddots        &   \\
					  &        & \Omega\Lambda & I
	\end{bmatrix}
	,
	\]
	and
	\[
	W_0:=(I\otimes I - FG\otimes \Omega\Lambda)^{-1} =
	\begin{cases}
		(-\Omega\Lambda)^{i'-j'}, &i'\ge j',\\
		0, &i'<j'.
	\end{cases}
	\]
	Then
	\[
	\|W_1\|_2
	= \|(I\otimes I - \Omega\Lambda\otimes FG)^{-1}\|_2
	= \|(I\otimes I - FG\otimes \Omega\Lambda)^{-1}\|_2
	= \|W_0\|_2.
	\]
	There are three cases to discuss in the following, and for whatever cases it is
	easy to see that $ \|W_0\|_1 = \|W_0\|_\infty=\|W_0e_1\|_1$
	and then $\|W_0\|_2\le\sqrt{\|W_0\|_1\|W_0\|_\infty}=\|W_0\|_1=\|W_0e_1\|_1$.
	\begin{enumerate}
		\item if $(\Lambda,\Omega)$ is of \typ{2}:
			\[
			(-\Omega\Lambda)^{i'-j'}=
			\begin{cases}
				(-1)^{i'-j'}, & i-j=i'-j',\\
				0, & i-j\ne i'-j'.
			\end{cases}
			\]
			Thus $\|W_1\|_2 \le \|W_0e_1\|_1 = \min\{n,n'\}$;
			and $\|W_2\|_2 \le \|W_0e_1 -e_1\|_1 = \min\{n,n'\}-1$.
		\item if $(\Lambda,\Omega)$ is of \typ{3}:
			\begin{equation}
				(-\Omega\Lambda)^{i'-j'}=
				\begin{cases}
					(-1)^{i'-j'}\lambda^{i'-j'-i+j}|\lambda|^{i-j}\tbinom{i'-j'}{i-j}, & i\ge j,\\
					0, & i<j,
				\end{cases}
				\quad \text{for $i'\ge j'$}.
				\label{eq:tmp:OmegaLambda-power-typ3}
			\end{equation}
			Thus
			\begin{align*}
				\|W_1\|_2 &\le \|W_0e_1\|_1
				\\&= 1+\sum_{i=1}^{n}\sum_{i'=2}^{n'}\left| (-1)^{i'-1}\lambda^{i'-i}|\lambda|^{i-1}\tbinom{i'-1}{i-1}\right|
				\\&= 1+\sum_{i=1}^{n}\sum_{i'=2}^{n'}|\lambda|^{i'-1}\tbinom{i'-1}{i-1}
				\\&\le 1+\sum_{i'=2}^{n'}|\lambda|^{i'-1}2^{i'-1}
				\\&= 1+\frac{2|\lambda|-(2|\lambda|)^{n'}}{1-2|\lambda|}
				\\&= 1+2|\lambda|\varphi(\frac{1}{2|\lambda|},n'-1),
			\end{align*}
			and
			\[
			\|W_2\|_2 \le \|W_0e_1 -e_1\|_1 = 2|\lambda|\varphi(\frac{1}{2|\lambda|},n'-1).
			\]
		\item if $(\Lambda,\Omega)$ is of \typ{4}:
			$ (-\Omega\Lambda)^{i'-j'}=\diag(R,\ol{R}) $
			and $R$ is of the form \eqref{eq:tmp:OmegaLambda-power-typ3}.
			Thus,
			\[
			\|W_1\|_2\le \max\{\|W_0e_1\|_1,\|\ol{W_0}e_1\|_1\} =  2|\lambda|\varphi(\frac{1}{2|\lambda|},n'-1)+1,
			\]
			and
			\[
			\|W_2\|_2 \le \max\{\|W_0e_1 -e_1\|_1,\|\ol{W_0}e_1 -e_1\|_1\} = 2|\lambda|\varphi(\frac{1}{2|\lambda|},n'-1).
			\]
	\end{enumerate}


{\small
\bibliographystyle{plain}
\bibliography{\TeXHOME/strings,\TeXHOME/relpert}
}

\end{document}